\documentclass[11pt,reqno]{amsart}
\usepackage{amssymb,bm,mathrsfs,stmaryrd, enumerate,bbm,amsmath,amsthm}
\usepackage[all]{xy}
\usepackage{hyperref}

\calclayout

\newcommand{\map}[1]{\xrightarrow{#1}}

\newcommand{\iso}{\cong}

\newcommand{\Hom}{\mathrm{Hom}}
\newcommand{\Aut}{\mathrm{Aut}}
\newcommand{\End}{\mathrm{End}}
\newcommand{\Spec}{\mathrm{Spec}}
\newcommand{\Q}{\mathbb Q}
\newcommand{\Z}{\mathbb Z}
\newcommand{\R}{\mathbb R}
\newcommand{\C}{\mathbb C}
\newcommand{\F}{\mathbb F}
\newcommand{\m}{\mathfrak m}
\newcommand{\A}{\mathbb A}
\newcommand{\co}{\mathcal O}
\newcommand{\LL}{\mathcal L}
\newcommand{\M}{\mathcal M}

\newcommand{\GL}{\mathrm{GL}}
\newcommand{\PGL}{\mathrm{PGL}}
\newcommand{\YZ}{\mathrm{YZ}}
\newcommand{\Nm}{\mathrm{Nm}}
\newcommand{\Tr}{\mathrm{Tr}}
\newcommand{\Div}{\mathrm{Div}}
\newcommand{\Frob}{\mathrm{Frob}}
\newcommand{\Bun}{\mathrm{Bun}}
\newcommand{\Fr}{\mathrm{Fr}}

\newcommand{\reg}{\mathrm{reg}}

\newcommand{\inv}{\mathrm{inv}}
\newcommand{\kk}{{\bm{k}}}
\newcommand{\cusp}{\mathrm{cusp}}
\newcommand{\Pic}{\mathrm{Pic}}

\newcommand{\K}{\mathbb K}
\newcommand{\J}{\mathbb J}
\newcommand{\I}{\mathbb I}

\newcommand{\G}{\mathbb G}
\newcommand{\Sht}{\mathrm{Sht}}
\newcommand{\Ch}{\mathrm{Ch}}
\newcommand{\Hk}{\mathrm{Hk}}
\newcommand{\Eis}{\mathrm{Eis}}

\begin{document}
\author{Ari Shnidman}
\title{Manin-Drinfeld cycles and derivatives of $L$-functions}


\address{Einstein Institute of Mathematics\\ Hebrew University of Jerusalem\\ Israel}
\email{ariel.shnidman@mail.huji.ac.il}

\begin{abstract}
We study algebraic cycles in the moduli space of $\PGL_2$-shtukas, arising from the diagonal torus.  Our main result shows that their intersection pairing with the Heegner-Drinfeld cycle is the product of the $r$-th central derivative of an automorphic $L$-function $L(\pi,s)$ and Waldspurger's toric period integral.  When $L(\pi,\frac12) \neq 0$, this gives a new geometric interpretation for the Taylor series expansion.   When $L(\pi,\frac12) = 0$, the pairing vanishes, suggesting higher order analogues of the vanishing of cusps in the modular Jacobian, as well as other new phenomena. 

Our proof sheds new light on the algebraic correspondence introduced by Yun and Zhang, which is the geometric incarnation of ``differentiating the $L$-function".  We realize it as the Lie algebra action of $e+f \in \mathfrak{sl}_2$ on $(\Q_\ell^2)^{\otimes 2d}$.  The comparison of relative trace formulas needed to prove our formula is then a consequence of Schur-Weyl duality.     
\end{abstract}

\maketitle

\theoremstyle{plain}
\newtheorem{theorem}{Theorem}[section]
\newtheorem{bigtheorem}{Theorem}

\newtheorem{proposition}[theorem]{Proposition}
\newtheorem{lemma}[theorem]{Lemma}
\newtheorem{corollary}[theorem]{Corollary}

\theoremstyle{definition}
\newtheorem{definition}[theorem]{Definition}

\theoremstyle{remark}
\newtheorem{remark}[theorem]{Remark}
\newtheorem{question}[theorem]{Question}

\numberwithin{equation}{section}

\setcounter{tocdepth}{1}


\section{Introduction}


Let $K/F$ be a quadratic extension of global function fields, corresponding to a double cover $\nu \colon Y \to X$ of smooth, projective, geometrically connected curves over $\kk = \F_q$.  We consider  cuspidal automorphic representations $\pi$ on $G = \PGL_{2,F}$.  Let $T$ be the torus $K^\times /F^\times$ over $F$.  For simplicity, we assume both $\pi$ and $K/F$ are  everywhere unramified.    
  
\subsection{Summary} 
Let $\Sht_G^{r}$ be the moduli stack over $\kk$ parameterizing $G$-shtukas with $r$ legs.  
Yun and Zhang define Heegner-Drinfeld cycles $[\Sht_T^{r}]_\pi \in \Ch_c^r(\Sht_G^{r})$ generalizing CM divisors on (Drinfeld) modular curves. 
In \cite[Cor.\ 1.4]{YZ2}, they relate the self-intersection of $[\Sht_T^r]_\pi$ to the product of the $r$-th central derivative of the normalized $L$-function $\mathscr{L}(\pi, s)$ and the central value of the twisted $L$-function $\mathscr{L}(\pi \otimes \eta,s)$:
\begin{equation}\label{yzformula2}
\mathscr{L}^{(r)}(\pi,1/2) \mathscr{L}(\pi \otimes \eta,1/2) \doteq \left\langle [\Sht_T^{r}]_\pi,[\Sht_T^{r}]_\pi\right\rangle_{\Sht_G^{r}}. 
\end{equation}
Here, $\eta$ is the quadratic character associated to $K$.  This is Waldspurger's formula \cite{Waldspurger} when $r = 0$, and a Gross-Zagier type formula when $r  = 1$.

%
%

In this paper, we define {\it Manin-Drinfeld cycles} $[\Sht_A^r]_\pi \in \Ch_c^r(\Sht_G^r)$ coming from $A$-shtukas, where $A \hookrightarrow G$ is the diagonal torus. 
They are generalizations of cuspidal divisors on the (Drinfeld) modular curve.  Our main result (Theorem \ref{thm:main intro}) has the following shape:
\begin{equation}\label{newformula}
\mathscr{L}^{(r)}(\pi,1/2) \int_{[T]}\phi(t) dt \doteq \left\langle [\Sht_A^r]_\pi,[\Sht_T^r]_\pi\right\rangle_{\Sht_G^r},
\end{equation}
for an appropriate spherical vector $\phi \in \pi$.  When $r = 0$, this formula amounts to the statement that $\mathscr{L}(\pi,s)$ can be written as a Mellin transform.  When $r = 1$, it is related to the function field version of the Manin-Drinfeld theorem \cite{Drinfeld,Manin}, that the cusps on the modular Jacobian are torsion.   For $r \geq 2$, our formula has no known analogue over number fields.  If $\mathscr{L}(\pi_K,1/2) \neq 0$, it shows that the Manin-Drinfeld cycles are non-vanishing and our formula gives a new expression for the Taylor series expansion of $\mathscr{L}(\pi,s)$.  If $\mathscr{L}(\pi_K,1/2)  = 0$, then we see that the Manin-Drinfeld cycle has trivial intersection with the Heegner-Drinfeld cycle, leading to an `alternative': either $[\Sht_A^r]_\pi$ is torsion, or $[\Sht_A^r]_\pi$ and $[\Sht_T^r]_\pi$ are linearly independent.

\subsection{Precise statement of results}
Let $Y_0 = X \coprod X \to X$ be the split double cover. 
The $F$-algebra of rational functions on $Y_0$ is $K_0=F\oplus F$.

There are natural  closed immersions 
\[
\xymatrix{
{  \widetilde{T} =\underline{\Aut}_{ \nu_{*} \co_{Y}  } ( \nu_{*} \co_{Y} )  }   \ar[d] &  {\hspace{5mm} } {  \widetilde{A} =\underline{\Aut}_{ \nu_{0*} \co_{Y_0}  } ( \nu_{0*} \co_{Y_0} )  }   \ar[d]\\  
{  \widetilde{G_1}= \underline{\Aut}_{\co_X} ( \nu_{*} \co_{Y} )  }& {\hspace{5mm} } {  \widetilde{G}= \underline{\Aut}_{\co_X} ( \nu_{0*} \co_{Y_0} )  } 
}
\] 
of group schemes over $X$.
Let $T\subset G_1$ and $A \subset G$ be the quotients by the central $\G_m$.
Then $G_1$ is Zariski-locally isomorphic to $\PGL_2$ and $G = \PGL_2$ over $X$. The  group scheme $T$ is a non-split torus,  while $A$ is the split diagonal torus in $\PGL_2$.  
 On  $F$-points, we have $T(F) = K^\times/F^\times$ and $A(F) \simeq F^\times$.

Let $\A$ be the adele ring of $F$, and $\mathbb{O}$ the subring of integral elements.
Define $U = G(\mathbb{O})$ and $U_1 = G_1(\mathbb{O})$. 
There is an isomorphism of spaces of cuspidal automorphic forms  \[\mathcal{A}_\cusp(G_1)^{U_1} \cong \mathcal{A}_\cusp(G)^{U}.\]  These  are finite dimensional $\C$-vector spaces, and the space on the right carries a natural action of the  Hecke algebra $\mathscr{H}$ of $\Q$-valued compactly supported $U$-bi-invariant functions on $G(\A)$.

We adopt the usual notation 
\[
[ T ] = T (F)\backslash T(\A)  \hspace{5mm} \mbox{and} \hspace{5mm}
[A]=A(F) \backslash A(\A).
\]
For any $\phi \in \mathcal{A}_\cusp(G)^{U}$, consider the toric period integrals
\[
\mathscr{P}_A(\phi,s) = \int_{[A]} \phi(a)|a|^{2s} da,
\]
and 
\[
\mathscr{P}_T(\phi) = \int_{[T]} \phi(t)dt.
\]
The Haar measures are chosen so that the volume of $A(\mathbb{O})$ and $T(\mathbb{O})$ is 1.

To precisely define the geometric side of (\ref{newformula}), recall from \cite{YZ} the stack $\Sht_{T}^r$ of $T$-shtukas with $r$ modifications, and the $2r$-dimensional $\kk$-stack $\Sht_{G}^r$ of $\PGL_2$-shtukas with $r$ modifications. The former parameterizes shtukas of line bundles on $Y$, while the latter parameterizes shtukas of rank 2 vector bundles on $X$.    
The $\kk$-stack $\Sht_T^r$ is proper of dimension $r$, and $\nu_*$ induces a finite morphism $\theta_T^r \colon \Sht_{T}^r \to \Sht_{G}^r$.  
Pushing forward the fundamental class along $\theta_T^r$ gives a class $[\Sht_{T}^r]\in  \Ch_c^r(\Sht_{G}^r)$ in the Chow group of compactly supported cycles. 

\begin{remark}
The definitions above require a choice of $\mu=(\mu_i)  \in \{\pm 1\}^r$ satisfying $\sum_{i=1}^r \mu_i =0$; in particular we assume that $r$ is even.   We suppress the choice of $\mu$  in the introduction.  
\end{remark}

Analogously, we define in Section \ref{s:intersection theory} a stack $\Sht_A^r$ parameterizing $A$-shtukas with $r$ modifications.  It is not of finite type over $\kk$, but can be written as a union 
\[\Sht_A^r = \bigcup_{d \geq 0} \Sht_A^{r,\leq d}\]
of stacks $\Sht_A^{r,\leq d}$ which are proper over $\kk$,
and which admit finite maps 
\[\Sht_A^{r,\leq d} \to \Sht_G^r.\]  
Define 
$[\Sht_{A}^{r,\leq d}]\in  \Ch_{c,r}(\Sht_{G}^r)$
as the push-forward of the fundamental class.

Fix $d \geq 0$, and denote by $\widetilde W^d_A, \widetilde W_T \subset \Ch_{c,r}(\Sht_{G}^r)$  the $\mathscr{H}$-submodule generated by the classes $[\Sht_{A}^{r,\leq d}]$, $[\Sht_T^r]$ respectively.
Restricting the intersection pairing on the Chow group defines a pairing
$
\langle \cdot , \cdot\rangle : \widetilde{W}^d_A \times \widetilde{W}_T \to \Q.
$
If we define  
\begin{align*}
W_A^d & = \widetilde{W}_A/  \{ c \in \widetilde{W}_A^d : \langle c, \widetilde{W}_T\rangle =0 \}  \\
W_T^d & = \widetilde{W}_T/  \{ c \in \widetilde{W}_T : \langle c, \widetilde{W}^d_A\rangle =0 \},
\end{align*}
this pairing descends to $ W_A^d \times W_T^d$, and we extend it to an $\R$-bilinear pairing
\[
\langle \,, \,  \rangle \colon W_A^d(\R)  \times W_T^d(\R) \to \R.
\]
We show in $\S 4$ that for $d \gg 0$, the space $W_T^d$ is independent of $d$, and for each $* \in \{A, T\}$, there a decomposition into isotypic components
\[
W_*^d (\R)= W^d_{*,\mathrm{Eis}} \oplus \left(\bigoplus _\pi W^d_{*, \pi}\right),
\] 
where the sum is over all unramified cuspidal $\pi$, and $\mathscr{H}$ acts on $W^d_{*, \pi}$ via $\lambda_\pi : \mathscr{H} \to \R$. 

Let
\[
[\Sht_{A}^{r,\leq d}]_\pi \in  W^d_{A, \pi} ,
\hspace{5mm} \mbox{and} \hspace{5mm}
[\Sht_{T}^r]_\pi \in  W^d_{T, \pi},\]
be the projections  of  $[\Sht_{A}^{r,\leq d}]\in W^d_A(\R)$ and $[\Sht_{T}^r] \in  W^d_T(\R)$.  These are the {\it Manin-Drinfeld} and {\it Heegner-Drinfeld} classes, respectively.  

Our main result is the following intersection formula.  Write $\mathscr{P}^{(r)}_A(\phi,s)$ for the $r$-th derivative of $\mathscr{P}_A(\phi,s)$.  We assume $d \gg 0$, so that the Manin-Drinfeld cycles are independent of $d$, and write $[\Sht_{A}^{r}]_\pi$ instead of $[\Sht_{A}^{r,\leq d}]_\pi$.

\begin{bigtheorem}\label{thm:main intro}
Let $\phi \in \pi^U$ be non-zero and let $r \geq 0$ be even.  Then
\[
\dfrac{\mathscr{P}^{(r)}_A(\phi,0) \mathscr{P}_T(\bar\phi)}{(\log q)^r  \langle \phi, \phi \rangle_\mathrm{Pet}} = \left\langle [\Sht_{A}^{r}]_\pi, [\Sht_{T}^r]_\pi\right\rangle_{\Sht_G^r}. 
\]  
\end{bigtheorem}
\begin{remark}
The left side is independent of the choice of $\phi$.
\end{remark} 
When $r = 0$, the formula is seen to be a tautology after unwinding the definition of the right hand side.  When $r  > 0$, it is helpful to interpret the formula in terms of $L$-functions.  Let $g$ be the genus of $X$.  Then the normalized $L$-function
\[\mathscr{L}(\pi,s) := q^{2(g-1)(s-\frac12)}L(\pi,s)\]
satisfies $\mathscr{L}(\pi,1-s) = \mathscr{L}(\pi,s).$  Moreover, for a suitably scaled $\phi \in \pi^U$, we have 
\[\mathscr{L}(\pi,2s + 1/2) = \mathscr{P}_A(\phi,s).\]  

The kind of information we learn from Theorem \ref{thm:main intro} depends on whether the base change $L$-function $\mathscr{L}(\pi_K,\frac12)$ vanishes at $s = 1/2$ or not.  Recall 
\[\mathscr{L}(\pi_K,s) = \mathscr{L}(\pi,s)\mathscr{L}(\pi \otimes \eta, s).\]
By Waldspurger's formula \cite[Rem.\ 1.3]{YZ}, $\mathscr{L}(\pi_K,\frac12) = 0$ if and only if $\int_{[T]} \bar\phi \, dt = 0$.

Thus, if $\mathscr{L}(\pi_K,\frac12) \neq 0$, Theorem \ref{thm:main intro} gives a geometric interpretation for the non-leading Taylor series coefficients of $\mathscr{L}(\pi,s)$, after dividing by the non-zero toric period integral.  To formulate this better, we consider the ratio with the leading term:  
\begin{bigtheorem}\label{ratio}
If $\mathscr{L}(\pi_K,\frac12) \neq 0$, then for even $r \geq 0$, we have $[\Sht_{A}^{r}]_\pi \neq 0$, and
\[\dfrac{\mathscr{L}^{(r)}(\pi,1/2)}{\mathscr{L}(\pi,1/2)} =  2^{-r}(\log q)^r \dfrac{\left\langle [\Sht_{A}^{r}]_\pi, [\Sht_{T}^r]_\pi\right\rangle_{\Sht_G^r}}{\left\langle [\Sht_{A}^{r}]_\pi, [\Sht_{T}^0]_\pi\right\rangle_{\Sht_G^0}}.\]
\end{bigtheorem}
\begin{remark}
That $[\Sht_{A}^{r}]_\pi \neq 0$ follows from the positivity of $\mathscr{L}^{(r)}(\pi,1/2)$ \cite[Thm. B.2]{YZ}.
\end{remark}
\begin{remark}
The precise version of (\ref{yzformula2}) implies 
\[\dfrac{\mathscr{L}^{(r)}(\pi,1/2)}{\mathscr{L}(\pi,1/2)} = (\log q)^r \dfrac{\left\langle [\Sht_{T}^r]_\pi, [\Sht_{T}^r]_\pi\right\rangle_{\Sht_G^r}}{\left\langle [\Sht_{T}^{0}]_\pi, [\Sht_{T}^0]_\pi\right\rangle_{\Sht_G^0}}.\]
This {\it suggests} that when $\mathscr{L}(\pi_K,\frac12) \neq 0$, we have $[\Sht_A^{r, \leq d}]_\pi = 2^rc_{\pi,K} [\Sht_T^r]_\pi$, where $c_{\pi,K}$ is an explicit non-zero ratio of period integrals.        
\end{remark}

If $\mathscr{L}(\pi_K,\frac12)  = 0$, Theorem \ref{thm:main intro} says nothing about $\mathscr{L}^{(r)}(\pi,1/2)$, but we still learn interesting information about algebraic cycles:  
\begin{bigtheorem}\label{consequence}
If $\mathscr{L}(\pi_K,\frac12) = 0$, then for even $r \geq  0$, we have
\[
\left\langle [\Sht_{A}^{r}]_\pi, [\Sht_{T}^r]_\pi\right\rangle_{\Sht_G^r} = 0. 
\]
\end{bigtheorem}

Write $r(\pi)$ for the order of vanishing of $\mathscr{L}(\pi,s)$ at $s = 1/2$.  If $r \geq r(\pi) \geq 0$ and $r(\pi \otimes \eta) = 0$, then  $[\Sht_T^r]_\pi \neq 0$ by (\ref{yzformula2}). As a consequence of Theorem \ref{consequence}, we have:
\begin{corollary}
If $r \geq r(\pi) > 0$ and $r(\pi \otimes \eta) = 0$, then either
\begin{enumerate}[$(i)$]
\item $[\Sht_{A}^{r}]_\pi = 0$, or
\item any lifts of $[\Sht_{A}^{r}]_\pi$ and $[\Sht_{T}^{r}]_\pi$ to the group $\Ch_{c,r}(\Sht_{G}^r)_\R$ are linearly independent.     
\end{enumerate}
\end{corollary}

\begin{remark}
There is a similar corollary, without any mention of lifts, for the cycle classes $[\Sht_A^{r, \leq d}]_{\pi, \lambda}$ and $[\Sht_T^r]_{\pi,\lambda}$ in $H^{2r}_c(\Sht_G^r \otimes_\kk \bar \kk, \Q_\ell(r))_{\pi, \lambda}$, as in \cite[1.5]{YZ}.       
\end{remark}

\subsection{Questions}
Which of $(i)$ or $(ii)$ actually holds?  Our expectation is that $[\Sht_A^r]_\pi \neq 0 $ precisely when $r \geq 2r(\pi)$.  
Indeed, our approach suggests a special value formula for the {\it self-intersection} of the Manin-Drinfeld cycle $[\Sht_A^r]_\pi$ in terms of the $r$th derivative of the {\it square} $\mathscr{L}(\pi,s)^2$, which would indeed imply $[\Sht_A^r]_\pi \neq 0$, for $r \geq 2r(\pi)$.  
If true, it would suggest an approach to the conjecture of Birch and Swinnerton-Dyer, for elliptic curves over function fields, that avoids the use of Heegner points or indeed any mention of the quadratic extension $K$.  

When the cycles $[\Sht_A^r]_\pi$ {\it do} vanish, one must wonder whether there is a related Euler system.  Indeed, Kato's Euler system is constructed from Siegel units which witness the torsion of the cuspidal divisors in the modular Jacobian.  Are there similar such functions for Manin-Drinfeld cycles? 

\subsection{Methods}
Yun and Zhang's proof of (\ref{yzformula2}) proceeds by {\it geometrizing} Jacquet's relative trace formula (RTF) comparison approach to Waldspurger's formula \cite{Jacquet}.  One side of this comparison involves traces of Frobenius on $\beta_*\Q_\ell$, where  $\beta \colon M_d \to A_d$ is a version of the Hitchin fibration (one for each integer $d \geq 0$).  A crucial insight in \cite{YZ} is that one can extend this approach to the case $r > 0$, using a certain natural correspondence 
\[\xymatrix{
& { \YZ_d }  \ar[dr]^{} \ar[dl]_{} \\
{  M_d  }  \ar[dr]_{ \beta }  & &  {   M_d    }  \ar[dl]^{  \beta }, \\
& { A_d } 
}\]
which we call the {\it Yun-Zhang} correspondence. The induced operator 
\[[\YZ_d] \colon \beta_*\Q_\ell \to \beta_*\Q_\ell\] 
plays the role of differentiation on the geometric side of the RTF comparison.  When $r > 0$, Yun and Zhang show that the traces of the operator $[\YZ_d]^r \circ \Frob$ match up with the $r$th derivative of certain traces on the $L$-function side.

Our proof of (\ref{newformula}) makes use of this RTF approach as extended in our work with Howard \cite{HS}. In fact, one can view (\ref{newformula}) as a degenerate version of our formula in \cite{HS} for the intersection of Heegner-Drinfeld cycles coming from {\it distinct} quadratic field extensions of $F$. In this work we suppose one of the quadratic $F$-algebras is split.  The RTF comparison becomes a tautology when $r =0$, but to prove (\ref{newformula}) for $r > 0$, we must dig deeper into the representation theory of the Hitchin fibration.  The key insight is a representation-theoretic interpretation of the operator $[\YZ_d]$.  In our setting, the local system $\beta_*\Q_\ell$ comes from the representation $(\Q_\ell^2)^{\otimes 2d}$ of the symmetric group $S_{2d}$.  This is also a representation of $\mathfrak{sl}_2$, and the Yun-Zhang operator is given by the action of the element $e+f = (\begin{smallmatrix} 0 & 1 \\ 1 & 0\end{smallmatrix}) \in \mathfrak{sl}_2$. Using Schur-Weyl duality and some computations in representation theory, this allows us to compare the two RTF's and prove (\ref{newformula}).                  

\subsection{Outline}
In Section 2, we define an analytic distribution on $\mathscr{H}$.  Following \cite{HS,YZ}, we relate it to $L$-functions on one hand,  and weighted traces of Frobenius along the Hitchin fibration, on the other.  One new input here is Lemma  \ref{eislemma}. The Manin-Drinfeld cycles are defined in Section 3.  We then use intersection pairings to define a geometric distribution, and relate it to traces of Frobenius along {\it the same} Hitchin fibration.  In Section 4, we work out the representation theory of this particular Hitchin fibration and relate it to the Yun-Zhang correspondence. A computation shows that the analytic and geometric distributions agree, and the main theorem follows quickly from this.   
\subsection{Notation}
\label{ss:notation}


$|X|$ is the set of closed points of $X$. The absolute value 
\[
| \cdot |= \prod_{x \in |X|} |\cdot|_x   \colon \A^\times \to \Q^\times
\] 
 sends the uniformizer $\pi_x \in F_x$ with residue field $\kk_x$ to
$ q^{ - [\kk_x \colon \kk]}$.
If $H$ is an algebraic group, Haar measure on $H(\A)$ is normalized so that $H(\mathbb{O})$ has volume 1.


\subsection{Acknowledgements}
The author thanks D.\ Kazhdan for several stimulating conversations on this topic. He also thanks D.\ Disegni, B.\ Howard, E.\ Lindenstrauss, K.\ Madapusi Pera, and Y.\ Varshavsky.  Special thanks go to S.\ Zemel for his helpful insight into the computations in Section 4.  The author was supported by the
Israel Science Foundation (grant No. 2301/20).


\section{Analytic distribution}
\label{s:analytic}



\subsection{Automorphic forms}
\label{ss:basic automorphic}


We recall some notation from \cite{HS}. Denote by $\mathcal{A}(G)$ the space of automorphic forms \cite[\S 5]{Borel-Jacquet} on $G(\A)$, and by 
$
\mathcal{A}_\cusp(G) \subset \mathcal{A}(G)
$
the subspace of cuspidal automorphic forms.  
The subspace of unramified ($U$-invariant) cuspforms is finite-dimensional, and admits a decomposition
\[
\mathcal{A}_\cusp(G)^{U} = \bigoplus_{\mathrm{unr.\, cusp.\, }\pi } \pi^{U}
\]
as a direct sum of lines, where the sum is over the  unramified cuspidal automorphic representations $\pi \subset \mathcal{A}_\cusp(G)$.

Denote by $\mathscr{H}$ the Hecke algebra of compactly supported functions $f: U \backslash G(\A) /U \to \Q$.
The $\mathscr{H}$-module of compactly supported unramified $\Q$-valued automorphic forms is denoted
\[
\mathscr{A} = C_c^\infty(  G(F) \backslash G(\A) / U ,\Q ).
\]
We let $\mathscr{A}_\C = \mathscr{A} \otimes \C$ denote  the corresponding complex space, so that 
\begin{equation}\label{harder}
\mathcal{A}_\cusp(G)^{U} \subset \mathscr{A}_\C \subset \mathcal{A} (G)^{U}.
\end{equation}

Following  \cite[\S 4.1]{YZ}, we view the Satake transform as a $\Q$-algebra surjection
$
a_{\mathrm{Eis}} : \mathscr{H} \to \Q  [\Pic_X(\kk)]^{\iota_\Pic},
$ 
for a particular involution $\iota_\Pic$ of $ \Q [\Pic_X(\kk)]$.   The Eisenstein ideal  is
\begin{equation}\label{eisenstein def}
\mathcal{I}^{\mathrm{Eis}} := \ker \big( a_\Eis  :  \mathscr{H} \to \Q  [\Pic_X(\kk)]^{\iota_\Pic}  \big).
\end{equation}
As in \cite[\S 7.3]{YZ}, define $\Q$-algebras
\begin{align*}
\mathscr{H}_\mathrm{aut} & = \mathrm{Image} \big( \mathscr{H}  \to \End_\Q (\mathscr{A}) \times \Q[\Pic_X(\kk)]^{\iota_\Pic} \big) \\
\mathscr{H}_\cusp &= \mathrm{Image}  \left(  \mathscr{H} \to \End_\C(  \mathcal{A}_\cusp(G)^{U} ) \right).
\end{align*}
The quotient map   $\mathscr{H}\to \mathscr{H}_\cusp$  factors through  $\mathscr{H}_\mathrm{aut}$,
and the resulting map
\begin{equation}\label{automorphic decomp}
\mathscr{H}_\mathrm{aut} \to \mathscr{H}_\cusp \times \Q  [\Pic_X(\kk)]^{\iota_\Pic} 
\end{equation}
is an isomorphism \cite[Lemma 7.16]{YZ}.

For each unramified cuspidal automorphic representation $\pi \subset \mathcal{A}_\cusp(G)$, denote by
\[
\lambda_\pi : \mathscr{H} \to \C
\] 
the character through which the Hecke algebra acts on the line $\pi^{U}$. 
 As in \cite[\S 7.5.1]{YZ}, the $\Q$-algebra $\mathscr{H}_\cusp$
is isomorphic to a finite product of number fields, and the product of all characters $\lambda_\pi$ induces an isomorphism
\[
 \mathscr{H}_\cusp \otimes \C  \iso \bigoplus_{    \mathrm{unr.\, cusp.\, }\pi   }  \C.
\]

The above the discussion holds word-for-word if $G$ is replaced by  $G_1$.

\begin{lemma}\label{lem:naive transfer}\cite[Lem.\ 3.3]{HS}
There is a canonical bijection
\[
G(F) \backslash G(\A) / U \to G_1(F) \backslash G_1(\A) / U_1.
\]
It induces an isomorphism
$\mathcal{A}(G)^{U} \iso \mathcal{A}(G_1)^{U_1}$,
respecting the subspaces of cusp forms.
\end{lemma}


\subsection{Definition of the distribution}
\label{ss:orbital notation}


The $X$-scheme
\[
\widetilde{J} = \underline{\mathrm{Iso}}_{\co_X}( \nu_{*} \co_{Y} ,  \nu_{0*} \co_{Y_0} ) 
\]
is both a left $\widetilde{G}$-torsor and a right $\widetilde{G}_1$-torsor.  Similarly $
J=\widetilde{J}  /\G_m$ is both a left $G$-torsor and a right $G_1$-torsor.
There are canonical identifications
\[
A(F) \backslash J(F) / T(F) = \widetilde{A}(F) \backslash \widetilde{J}(F) / \widetilde{T}(F) = K_0^\times \backslash \mathrm{Iso}( K , K_0 ) / K^\times.
\]
Thus, \cite[\S 2]{HS} allows us to define the invariant map 
\begin{equation}\label{orbit invariant}
A(F) \backslash J(F) / T(F) \map{\inv}  \{ \xi \in K : \Tr_{K/F}(\xi)=1 \}.
\end{equation}
Since $K$ and $K_0$ are non-isomorphic, the map $\mathrm{inv}$ is a bijection. 
\begin{remark}
Here is an explicit description of $\mathrm{inv}$, when $\mathrm{char} \, \kk \neq 2$.  Choose $F$-algebra embeddings $\alpha_ 1 \colon K \hookrightarrow M_2(F)$ and $\alpha_2 \colon K_0 \hookrightarrow M_2(F)$.  Let $e = \alpha_2(1,0)$ and $f = \alpha_2(0,1)$ be the image of the two idempotents.  Then $M_2(F) = e\alpha_1(K) + f\alpha_1(K)$.  If 
\[g \in K_0^\times \backslash \mathrm{Iso}( K , K_0 ) / K^\times \simeq A(F) \backslash G(F)/T(F)\] is represented by $e\alpha + f \beta \in G(F)$,  then $\mathrm{inv}(g) = \frac{2\alpha\bar\beta}{\Tr(\alpha \bar \beta)}$.       
\end{remark}

\begin{lemma}\label{lem:coset transfer}\cite[Lem.\ 3.4]{HS}
There is a canonical homeomorphism
\begin{equation}\label{coset switch}
 U \backslash J(\A) / U_1 \iso U \backslash G(\A) /U.
\end{equation}
\end{lemma}

Now fix $f\in \mathscr{H}$.  Use the bijection of Lemma \ref{lem:coset transfer} to view $f$ as a function 
\[
f : U \backslash J(\A) / U_1 \to \Q,
\]
and  define a function on $G(\A) \times G_1(\A)$ by
 \begin{equation}\label{the kernel}
\K_f(g,g_1) = \sum_{\gamma \in J(F)} f(g^{-1}\gamma g_1).
\end{equation}
   
 Recall the notation $[ T ] = T (F)\backslash T(\A)$ 
from the introduction,
and recall the normalization of Haar measures of \S \ref{ss:notation}.
Define a distribution on $\mathscr{H}$ by 
\begin{equation}\label{J distribution def}
\J(f,s) = \int^\reg_{[A] \times [T]} \K_f(a, t)\,  |a|^{2s}  da \, dt.
\end{equation}
Here,  $|\cdot| : A(\A) \to \R^\times$ is the homomorphism  $\left| \left(\begin{smallmatrix} a_1 & \\ & a_2 \end{smallmatrix} \right)\right| = |a_1/a_2|$.  

The integral in (\ref{J distribution def}) need not converge absolutely, so we regularize it.  First define 
\[
A(\A)_n = \left\{a \in A(\A) :   |a| = q^{-n} \right\} 
\]
and $[A]_n = A(F) \backslash A(\A)_n$, and set
\begin{align}\label{partial distribution}
\J_n(f,s) & = \int_{[A]_n \times [T]} \K_f(a,t) |a|^{2s} \,  da \, dt \\
& =  q^{-2ns} \int_{[A]_n \times [T]} \K_f(a,t)  \,  da \, dt \nonumber.
\end{align}
This integral is absolutely convergent, by compactness of $[A]_n$ and $[T]$.

\begin{proposition}\label{prop:convergence}
The integral $\J_n(f,s)$ vanishes for $|n|$ sufficiently large.  
\end{proposition}

\begin{proof}
As in \cite[Prop.\ 3.7]{HS}.
\end{proof}

Using Proposition \ref{prop:convergence},  the regularized integral  (\ref{J distribution def}) is defined as
\[
\J(f,s) = \sum_{n \in \Z} \J_n(f,s).
\] 
This is a Laurent polynomial in $q^s$. Define
\begin{equation} \label{trunc J}
\J_n(\gamma, f,s)  = \int_{[A]_n \times [T]} \K_{f,\gamma}(a,t) |a|^{2s} \,  da \, dt,
\end{equation}
and
\[
\J(\gamma,f,s) =\sum_{n\in \Z} \J_n(\gamma,f,s),
\]
 so that there are decompositions
\begin{equation}\label{J decomp}
\J(f,s)   =   \sum_{\gamma \in A(F) \backslash J(F)/T(F)}  \J(\gamma, f,s)  = \sum_{\substack{ \xi\in K \\ \Tr_{K/F}(\xi) =1  }} \J( \xi ,f,s).
\end{equation}
In the final expression,  we have used (\ref{orbit invariant}) to define
\[
\J(\xi,f,s)= \J(\gamma,f,s)
\] 
for the unique double coset $\gamma \in A(F) \backslash J(F)/T(F)$ satisfying $\inv(\gamma)=\xi$.


\subsection{Spectral decomposition}
\label{ss:spectral}


Define for any $\phi \in \mathcal{A}_\cusp(G)^{U}$, the period integral 
\[
\mathscr{P}_A(\phi , s ) = \int_{[A]} \phi(a) |a|^{2s} \, da.
\]
This integral is absolutely convergent for all $s\in \C$.
Using Lemma \ref{lem:coset transfer} to view $\phi \in \mathcal{A}_\cusp(G_1)^{U_1}$, define another period integral
\[
\mathscr{P}_T(\phi ) = \int_{[T]} \phi(t)  \, dt.
\]
As $[T]$ is compact, this integral is also absolutely convergent.

Recall the Eisenstein ideal $\mathcal{I}^\Eis \subset \mathscr{H}$ of (\ref{eisenstein def}).

\begin{proposition}\label{J pi decomp}
Every $f\in\mathcal{I}^\Eis$ satisfies 
\begin{equation}\label{spectral}
\J (f,s) =  \sum_{\mathrm{unr.\, cusp.\,}\pi }  \lambda_\pi(f)   \frac{  \mathscr{P}_A(  \phi ,s) \mathscr{P}_T(\overline{\phi} )}{\langle \phi, \phi \rangle },
\end{equation}
where the sum is over all unramified cuspidal automorphic representations $\pi \subset \mathcal{A}_\cusp(G)$, and $\phi \in \pi^{U}$ is any nonzero vector.
Moreover, $\J(f,s)$ only depends on the image of $f$ under the quotient map  $\mathscr{H} \to \mathscr{H}_\mathrm{aut}$.
\end{proposition}

\begin{proof}
View  (\ref{the kernel}) as a function on $G(\A) \times G(\A)$, and invoke the decomposition 
\[
\K_f (x,y) = \K_{f,\cusp} (x,y)  +  \K_{f,\mathrm{sp}} (x,y)  
\]
of \cite[Theorem 4.3]{YZ}, to convert all three terms back into functions on  $G(\A) \times G_1(\A)$.  The result is a decomposition
\begin{align*}
\K_f (g,g_1)   &  =  \sum_{\mathrm{unr.\, cusp.\,}\pi } \lambda_\pi(f)  \cdot  \frac{ \phi(g) \overline{\phi(g_1)} }{ \langle \phi ,\phi \rangle}  \\
&\quad  + \sum_{\mathrm{unr.\, quad.\,}\chi } \lambda_\chi(f)   \cdot  \chi(\det(G))  \cdot   \chi(\det(g_1))  . 
\end{align*}
The first  sum is over all unramified cuspidal representations $\pi$,  and $\phi \in \pi^{U}$ is any nonzero vector.
The second sum is over all unramified quadratic characters 
\[
\Pic(X) \iso  F^\times \backslash \A^\times / \mathbb{O}^\times  \map{\chi} \{ \pm 1\},
\]
and 
\[
\lambda_\chi(f) = \int_{G(\A)} f(g) \chi(\det(g))\, dg.
\]

The distribution (\ref{partial distribution}) now decomposes as
\[
\J_n(f,s) =   \sum_{\mathrm{unr.\, cusp.\,}\pi }  \J_n^\pi(f,s) 
+  \sum_{\mathrm{unr.\, quad.\,}\chi } \J_n^\chi(f,s) ,
\]
where we have set
\[
\J^\pi _n(f,s) =  \frac{ \lambda_\pi(f)  }{  \langle \phi ,\phi \rangle }
\left( \int_{[A]_n } \phi(a)   |a|^{2s} \,  da  \right)
\left( \int_{ [T]}   \overline{\phi(t)}  \,   dt \right)
\]
and
\begin{equation}\label{special}
\J^\chi _n(f,s)  = \lambda_\chi(f)  
\left(   \int_{[A]_n } \chi(\det(a))  |a|^{2s} \,  da    \right)
\left(  \int_{ [T]}   \chi(\det(t))  \,   dt \right) .
\end{equation}

Next we show that $\J_n^\chi(f,s)=0$ for all such $\chi$.  Note that when $\chi = 1$, {\it both} toric integrals in (\ref{special}) are non-zero, so the proof from \cite{HS} does not carry over. 
The vanishing in all cases follows from:
\begin{lemma}\label{eislemma}
If $f \in \mathcal{I}^\Eis$ and $\chi$ is unramified, then $\lambda_\chi(f) = 0$. 
\end{lemma}

\begin{proof}
Let $B \subset G$ be the Borel subgroup of upper triangular matrices. Following \cite[\S4]{YZ}, we consider the right translation representation $\rho_\chi$ of $G(\A)$ on the space $V_{\chi}$ of functions 
\[\phi \colon G(\A) \to \C\]
such that $\phi(bg) = \chi(b)\phi(g)$ for all $b \in B(\A)$, $g \in G(\A)$.
The space $V_{\chi}$ is canonically identified (by restriction) with a space of functions on $U$.  The latter space carries an inner product 
\[(\phi, \phi') = \int_{U} \phi(u)\bar \phi'(y) du.\]
Now let $\phi = \bm{1}_{U}$.  Since $f \in \mathcal{I}^\Eis$, we have \cite[\S4]{YZ} 
\[(\rho_\chi (f)\phi, \phi) = \mathrm{tr}\, \rho_\chi(f) = \chi(a_\Eis(f)) = 0.\]
On the other hand, we compute
\begin{align*}
(\rho_\chi (f)\phi, \phi) &= \int_{U} \int_{g = bu' \in G(\A)} f(u^{-1}g)\chi(\det b) dg du\\
 &= \int_{G(\A)} f(g)\chi(\det g) dg \\
 &= \lambda_\chi(f).
\end{align*}
We have used that $\chi$ is unramified, $f$ is spherical, and $U$ has volume 1.  We conclude that $\lambda_\chi(f) = 0$.
\end{proof}
 
 We therefore have
 \[
\J_n(f,s) =    \sum_{\mathrm{unr.\, cusp.\,}\pi }   \J_n^\pi(f,s) ,
\]
and (\ref{spectral}) follows by summing both sides over $n$.

For the final claim, suppose $f$ has trivial image under $\mathscr{H}\to\mathscr{H}_\mathrm{aut}$.  This implies that $f$ annihilates $\mathscr{A}_\C$, 
and lies in $\mathcal{I}^\Eis$.  The first inclusion in  (\ref{harder}) implies that  $\lambda_\pi(f)=0$ for all $\pi$, and so $\J(f,s)=0$ by (\ref{spectral}).
\end{proof}


\subsection{Geometric expression}
\label{ss:analytic traces}


Fix $d \geq 0$, and let $i$ be an integer in the range $0 \leq i \leq 2d$. Recall from \cite[\S 3]{HS} the commutative diagram of $\kk$-schemes
\begin{equation}\label{fundamental 2}
\xymatrix{
 {  N_{ (i,2d-i) }  }   \ar[d]_{\beta_i}  \ar[r]_{\delta} &  { \Sigma_{i, 2d-i}(Y) } \ar[d]^{\otimes}  \\
 {    A_d  }  \ar[d]_{\Tr}   \ar[r]_{\nu^\sharp} &    \Sigma_{2d}(Y)  \\
  {    \Sigma_d(X)  } 
}
\end{equation}
in which the square is cartesian.
   We briefly recall the definitions of these  schemes and maps in terms of their $S$-points, for any scheme $S$.

First, $\Sigma_d(X)(S)$  is the set of isomorphism classes of pairs $(\Delta, \zeta )$ of
\begin{itemize}
\item
 a line bundle $\Delta$ on $X_S=X\times_\kk S$ of degree $d$,
\item
a nonzero section $\zeta \in  H^0(X_S, \Delta)$.
\end{itemize}
We have a canonical isomorphism
\begin{equation}\label{symmetric id}
\Sigma_d(X) \iso \mathrm{Sym}^d(X) \iso S_d\backslash X^d,
\end{equation}
and $\Sigma_d(X)$ is a smooth projective $\kk$-scheme. We also set 
\[\Sigma_{i, 2d-i}(Y) = \Sigma_{i}(Y) \times_\kk \Sigma_{2d-i}(Y),\] parameterizing effective divisors of bidegree $(i,2d-i)$ on $Y \coprod Y$. 

Next,  $A_d(S)$ is the set of isomorphism classes of  pairs  $(\Delta, \xi )$ consisting of
\begin{itemize}
\item
a  line bundle $\Delta$ on $X_{S}$ of degree $d$,
\item
a section $\xi \in H^0( Y_S , \nu^*\Delta)$ with nonzero trace 
\[
\mathrm{Tr}_{Y/X}(\xi) = \xi+\xi^{\sigma_1} \in H^0(X_S,\Delta).
\] 
\end{itemize}
The arrows in (\ref{fundamental 2}) emanating from $A_d$ are
\[
\Tr (\Delta,\xi) = (\Delta , \Tr_{Y/X} (\xi) ) \quad\mbox{ and }\quad  \nu^\sharp (\Delta, \xi)  = (\nu^*\Delta ,\xi).
\]
$A_d$ is a quasi-projective $\kk$-scheme.

Finally, $\widetilde{N}_{(i,2d-i)}(S)$ is the groupoid of triples $(\mathcal{M},  \mathcal{L}, \phi)$ consisting of
\begin{itemize}
\item line bundles $\mathcal{M} = (\mathcal{M}', \mathcal{M}'')\in \mathrm{Pic}(Y_{0S})$ and $\mathcal{L}\in \mathrm{Pic}(Y_S)$ satisfying 
\[
2\deg(\mathcal{M}') - i=  \deg(\mathcal{L})  =2\deg(\mathcal{M}'') - (2d-i),
\]
\item a morphism  $\phi  :   \nu_*\mathcal{L} \to \mathcal{M}' \oplus \mathcal{M}'' $ of rank 2 vector bundles on $X_S$
with nonzero determinant.
\end{itemize}
The Picard group $\Pic(X_S)$ acts on $\widetilde{N}_{i,2d-i}(S)$ by simultaneous twisting, and
the quotient
\[
N_{ (i,2d-i) } = \widetilde{N}_{(i,2d-i)} /\Pic_X 
\]
is a scheme by \cite[3.12]{HS}.  
The map $\delta$ sends $(\mathcal{M}, \mathcal{L}, \phi)$ to $(\mathrm{div}(\bm{a}), \mathrm{div}(\bm{d)})$, where 
$\bm{a}$
and 
$\bm{d}$
are the diagonal matrix entries in the map $\nu^*\phi \colon \mathcal{L} \oplus \mathcal{L}^\sigma \to \nu^*\M' \oplus \nu^*\M''$.  

\begin{proposition}\label{prop:Nsmooth}\cite[3.13]{HS}
Let $g$ and $g_1$ be the genera of $X$ and $Y$, respectively.
\begin{enumerate}
\item
The morphisms $\beta_i$ and $\otimes$ in (\ref{fundamental 2}) are finite.   
\item
 If $d\ge 2 g_1-1 $ then $N_{(i,2d-i)}$ is  smooth over $\kk$ of  dimension $2d -g +1$.
\end{enumerate}
\end{proposition}
Now let $D$ be an effective  divisor on $X$ of degree $d$.   The constant function $1$ defines a global section of $\co_X(D)$, and hence a point
$(\co_X(D), 1) \in \Sigma_d(X)(\kk)$.  Define $A_D$ as the fiber product
\[
\xymatrix{
{ A_D } \ar[rr] \ar[d] & &  { A_d } \ar[d] \\
{ \Spec(\kk) } \ar[rr]^{  ( \co_X(D) ,1 )  } & & { \Sigma_d(X) .}
}
\]
Then there is a canonical bijection
\begin{equation}\label{invariant domain}
A_D(\kk) \iso  \left\{ \xi \in K : \begin{array}{c} \Tr_{K/F}(\xi) =1 \\ \mathrm{div}(\xi) + \nu^*D \ge 0   \end{array}  \right\} .
\end{equation}

The Hecke algebra $\mathscr{H}$ has a $\Q$-basis $\{ f_D \}$ indexed by the effective divisors $D\in \Div(X)$, and defined as follows (see also \cite[\S 3.1]{YZ}).  Let $S_D$ be the image of the set 
\[\left\{ M \in \mathrm{Mat}_2(\mathbb{O})\colon \mathrm{div}(\det M) = D\right\}\]
in $\PGL_2(\A) = G(\A)$.  Then $f_D \colon U\backslash G(\A)/U \to \Q$ is the characteristic function of $S_D$.

We are now ready to give a geometric interpretation of the orbital integral $\J(\xi,f_D,s)$ appearing in (\ref{J decomp}).  Using Lemma \ref{lem:coset transfer}, we regard $f_D$  as a compactly supported function 
\begin{equation}\label{D hecke}
f_D : U \backslash J(\A) / U_1 \to \Q.
\end{equation}
Let $\ell$ be any prime different from the characteristic of $\kk$.   The following theorem is proved exactly as in \cite[3.17]{HS}, this time with a trivial local system.    
\begin{theorem}\label{prop:geometric orbital}
Fix  $\xi\in K$ with $\Tr_{K/F}(\xi)=1$,  and view $A_D(\kk)$ as a subset of  $K$ via $(\ref{invariant domain})$.
\begin{enumerate}
\item If $\xi \not\in A_D(\kk)$  then $\J( \xi , f_D,s) = 0$.
\item If  $\xi \in A_D(\kk)$  then 
 \[
 \J( \xi , f_D,s) =
\sum_{ i = 0}^{2d}
 q^{  2(i-d)s }  \cdot
 \Tr \big(\Frob_\xi  ; \,   \big( \beta_{i*}{\Q_\ell}  \big)_{\bar \xi} \big),
\]
where $\bar{\xi}$ is a geometric point above $\xi:\Spec(\kk) \to A_D \hookrightarrow A_d$.
\end{enumerate}       
\end{theorem}

\section{Geometric distribution}
\label{s:intersection theory}


Fix an  integer $r\ge 0$, and an $r$-tuple $\mu=(\mu_1,\ldots, \mu_r) \in \{ \pm 1\}^r$ satisfying  the parity condition
$\sum_{i=1}^r \mu_i =0.$  In particular,  $r$ is even.


\subsection{Heegner-Drinfeld and Manin-Drinfeld cycles}
\label{ss:cycles}


We recall some notation from \cite{HS,YZ}.
Recall that $G = \PGL_2$ over $X$.

Let  $\Bun_{G}$ be the algebraic stack parametrizing $G$-torsors on $X$, and let 
 $\Hk_{G}^\mu$  be the Hecke stack parameterizing $G$-torsors on $X$ with $r$ modifications of type $\mu$.  
 It comes equipped with morphisms
 \[
 p_0,\ldots, p_r : \Hk_{G}^\mu \to \Bun_{G}
 \]
 and $p_X : \Hk_{G}^\mu \to  X^r$.  For the definitions, see \cite[\S 5.2]{YZ}.

The moduli stack $\Sht_G^\mu$ of $G$-shtukas of type $\mu$ sits in the cartesian diagram
\[
\xymatrix{\Sht^{\mu}_{G}\ar[d]\ar[rr] && \Hk^{\mu}_{G}\ar[d]^{(p_{0},p_{r})}\\
\Bun_{G}\ar[rr]^{(\mathrm{id},\Fr)} && \Bun_{G}\times \Bun_{G} .}
\]  
It is a Deligne-Mumford stack, locally of finite type over $\kk$, and the  morphism 
\[
\pi_{G} : \Sht_{G}^\mu \to X^r
\]
induced by $p_X$  is  separated and smooth of relative dimension $r$.  

The \'etale double covers $\nu_0:Y_0\to X$ and $\nu:Y\to X$ determine tori $A$ and $T$, both rank 1 over $X$. 
Let $\mathrm{Bun}_{T}$ be the moduli stack of $T$-torsors on $X$. 
Denote by   $\Hk_{T}^\mu$ the Hecke stack parameterizing $T$-torsors with $r$ modifications of type $\mu$.
It comes with morphisms
\[
p_1,\ldots , p_r:  \Hk_{T}^\mu \to \mathrm{Bun}_{T},
\]
and   $p_{Y} \colon \Hk_{T}^\mu \to Y^r$.
 See \cite[\S5.4]{YZ} for the definitions.

The stack of $T$-shtukas of type $\mu$ is defined by the cartesian square
\begin{equation}\label{Sht_T definition}
\xymatrix{\Sht^{\mu}_{T}\ar[d]\ar[rr] && \Hk^{\mu}_{T}\ar[d]^{(p_{0},p_{r})}\\
\mathrm{Bun}_{T}\ar[rr]^{(\mathrm{id},\Fr)} && \mathrm{Bun}_{T}\times \mathrm{Bun}_{T} . }
\end{equation}
It is Deligne-Mumford over $\kk$, and the morphism
\[
\pi_{T} : \Sht_{T}^\mu \to Y^r
\]
induced by $p_{Y}$ is  finite \'etale.  Thus, $\Sht_{T}^\mu$ is smooth and proper over $\kk$, and of dimension $r$.

For the diagonal torus $A$, we similarly define the stack $\Sht^{\mu}_A$ by the cartesian square
\begin{equation}\label{Sht_A definition}
\xymatrix{\Sht^{\mu}_{A}\ar[d]\ar[rr] && \Hk^{\mu}_{A}\ar[d]^{(p_{0},p_{r})}\\
\mathrm{Bun}_{A}\ar[rr]^{(\mathrm{id},\Fr)} && \mathrm{Bun}_{A}\times \mathrm{Bun}_{A} . }
\end{equation}
Here, $\Bun_A$ is the stack of $A$-torsors, whose $S$-points parameterize line bundles $\mathcal{M} = (\mathcal{L}_1,\mathcal{L}_2)$ on $Y_{0S}$, modulo simultaneous twisting by $\LL \in \Pic_{X_S}$. 

More concretely, let $\widetilde \Sht_{A}^\mu$ be the stack whose $S$-points is the groupoid of $(\mathcal{M}, y_1,\ldots, y_r, \iota)$, where $\mathcal{M} = (\mathcal{L}_1,\mathcal{L}_2)$ is a line bundle on $Y_{0S}$, the $y_i \colon S \to Y_0$ are $S$-points of $Y_0$, and $\iota$ is an isomorphism
\[\iota \colon \mathcal{M}^\Fr \simeq \mathcal{M}\left(\sum_{i = 1}^r \mu_{i} \Gamma_{y_i}\right).\] 
Here, $\Gamma_{y_i}$ is the graph of $y_i$.  Then $\Sht_A^\mu = \widetilde\Sht_A^\mu/\Pic_X(\kk)$, where $\Pic_X(\kk)$ acts by simultaneous twisting on $\mathcal{L}_1$ and $\mathcal{L}_2$.  
  
Let $\pi_{A} : \Sht_{A}^\mu \to Y_0^r$ be the morphism which remembers the legs $y_i \in Y_0$ of the $A$-shtuka.
\begin{lemma}
The morphism $\pi_{A}$ is a $\Pic_X(\kk)$-torsor.
In particular, $\Sht_{A}^\mu$ is a smooth Deligne-Mumford stack over $\kk$,  {\it locally} of finite type.  
\end{lemma}
\begin{proof}
The proof is exactly as for $T$-shtukas \cite[Lem.\ 5.13]{YZ}, using that 
\[\Pic_X(\kk) \simeq (\Pic_X(\kk) \times \Pic_X(\kk))/\Delta(\Pic_X(\kk)).\]  
\end{proof}
For each $d \geq 0$, the open substack $\Sht_A^{\mu,\leq d} \subset \Sht_A^\mu$ consisting of those $(\mathcal{M}, (y_i) , \iota)$ with 
\[|\deg(\mathcal{M})| := |\deg(\mathcal{L}_2) - \deg(\mathcal{L}_1)| \leq d,\]
is proper over $\kk$.
In fact, each of the substacks 
\[\Sht_A^{\mu,d} =  \langle (\mathcal{M},(y_i), \iota) \colon |\deg(\mathcal{M})| = d\rangle,\]
 are themselves proper and closed. 

Push-forward of line bundles induces proper morphisms 

\[
\xymatrix{
{   \Sht_{A}^{\mu,\leq d} } \ar[dr]_{\theta_A^{\mu,\leq d}} &   &   {  \Sht_{T}^\mu  } \ar[dl]^{\theta_T^\mu}  \\
 & {  \Sht_{G}^\mu  } 
}
\]
since $\Sht_G^\mu$ is separated.  
We therefore obtain two classes 
\[[\Sht_T^\mu ], [\Sht_A^{\mu,\leq d}] \in  \Ch_{c,r} (  \Sht_{G}^\mu ),\]
by pushing forward the corresponding fundamental classes.  

\subsection{Geometric distribution}
There is an intersection pairing
\[
\langle\cdot,\cdot\rangle :  \Ch_{c,r} (  \Sht_{G}^\mu ) \times  \Ch_{c,r} (  \Sht_{G}^\mu ) \to \Q,
\]
as in  \cite[\S A.1]{YZ}.  Recall the Hecke algebra $\mathscr{H}$ of  \S \ref{ss:basic automorphic} and its action $*$ on $\Ch_{c,r} (  \Sht_{G}^\mu )$ \cite[\S 5.3]{YZ}.  Recall also \cite[\S 7]{YZ} that $\Sht_G^\mu$ can be written as the increasing union of open substacks of finite type: 
\[\Sht_G^\mu = \bigcup_{\underline d \in \mathcal{D}} \Sht_G^{\leq \underline d}.\]
Here, $\mathcal{D}$ is the set of functions $\Z/r\Z \to \Z$, which is partially-ordered by pointwise comparison.  The substack $\Sht_G^{\leq \underline d}$ parameterizes $G$-shtukas $\mathcal{E}$ such that the vector bundle $p_i(\mathcal{E})$ has index of instability less than or equal to $d(i)$, for all $i = 0,\ldots, r$.

For any $f\in\mathscr{H}$ define 
\begin{equation}\label{I distribution def}
\I_r(f) = \langle  [\Sht_{A}^{\mu,\leq d}] ,  f * [\Sht_{T}^\mu]  \rangle \in \Q,
\end{equation}
where $d$ is any integer with the property that 
\[f*[\Sht_T^\mu] \in \Ch_{c,r}(\Sht_G^\mu) = \lim_{\substack{\longrightarrow \\ d \in \mathcal{D}}}\Ch_r(\Sht_G^{\mu,\leq \underline d})\] is supported on $\Sht_G^{\mu, \leq d - r}$. In the last bit of notation, we view $d - r$ as a constant function in $\mathcal{D}$.  Note that this intersection number is independent of the choice of such $d$, since  
\[[\Sht_A^{\mu, \leq d + n}] = [\Sht_A^{\mu, \leq d}] + \sum_{i = 1}^n[\Sht_A^{\mu, d+i}]\]
for any $n \geq 0$, and $\langle [\Sht_A^{\mu, d+i}], f*[\Sht_T^\mu]\rangle = 0$. In particular, it follows that the function $\I_r(f)$ is additive, and hence defines a distribution on $\mathcal{H}$.  


\subsection{Intersections as traces}
\label{ss:geometric intersection}


Fix $d \geq 0$, and recall from $\S 2$ the morphism  
\[ \beta_i \colon N_{(i,2d-i)} \to A_d,\] 
defined for each non-negative $i \leq 2d$.  Define 
\[N_d = \coprod_{i = 0}^{2d} N_{(i,2d-i)}.\] 
Then $N_d = \widetilde N_d /\Pic_X$, where $\widetilde N_d$ is the moduli stack of triples 
\[(\M \in \Pic_Y, \LL \in \Pic_{Y_0}, \phi \colon \nu_*\LL \to \nu_{0*} \M)\] such that  $\phi$ has determinant of degree $d$.      
Write $\beta \colon N_d \to A_d$ for the union of the $\beta_i$.  

Let $D$ be an effective divisor of degree $d$.  To relate $\I_r(f_D)$ to the local system $\beta_*\Q_\ell$, we define the Yun-Zhang  correspondence 
\begin{equation}\label{hecke one paw}
\xymatrix{
& { \YZ_d }  \ar[dr]^{\gamma_1} \ar[dl]_{\gamma_0} \\
{  N_d  }  \ar[dr]_{ \beta }  & &  {   N_d    }  \ar[dl]^{  \beta } \\
& { A_d } 
}
\end{equation}
as follows.  Let $\widetilde{\YZ}_{d}$ be the stack whose $S$-points is the groupoid of 
\begin{itemize}
\item 
pairs of points $(\mathcal{M}_0,\mathcal{L}_0,\phi_0)$ and  $(\mathcal{M}_1,\mathcal{L}_1,\phi_1)$ in $\widetilde N_d(S)$,
\item
one  $S$-point 
$(y_0,y_1)$ of  $Y_{0S}\times_{X_S} Y_{S}$,
\item
and injective line bundle maps  $s_0 : \mathcal{M}_0  \to  \mathcal{M}_1$ and $s_1 : \mathcal{L}_0  \to  \mathcal{L}_1$.
\end{itemize}
The cokernels of $s_i$ are required to be invertible sheaves on the graphs of $y_i$.  Moreover,  we require that the diagram
\[
\xymatrix{
{   \nu_{*} \mathcal{L}_0 } \ar[r]^{s_1}     \ar[d]_{\phi_0}  &  {   \nu_{*} \mathcal{L}_1 }   \ar[d]^{\phi_1}     \\
  {   \nu_{0*} \mathcal{M}_0  }  \ar[r]_{ s_0}    &   {   \nu_{0*} \mathcal{M}_1  } 
}
\]
of $\co_{X_S}$-modules commutes.  Then $\Pic_X$ acts on $\widetilde{\YZ}_{d}$ by simultaneously twisting, and we define $\YZ_d = \widetilde{\YZ}_{d}/\Pic_X$.

Let $W = Y \coprod Y$.  Also let $\Sigma_{2d}(W)$ be the moduli stack of pairs $(\mathcal{K}, \bm{a})$, consisting of a line bundle $\mathcal{K}$ of degree $2d$ on $W$, together with a global section $\bm{a}$.  Recalling the spaces $\Sigma_{i,j}(Y)$ introduced in \S2, we have 
\[\Sigma_{2d}(W) = \coprod_{i = 0}^{2d} \Sigma_{i, 2d-i}(Y).\]
Using the top horizontal arrow of (\ref{fundamental 2}),  we may realize $\YZ_d$  as the pullback of a correspondence on $\Sigma_{2d}(W)$.  For this, define the following automorphisns $\tau_i$ of $W$ over $X$:
\begin{enumerate}
\item $\tau_1$ is $\sigma$ on $Y \coprod \varnothing$ and the identity on $\varnothing \coprod Y$. 
\item $\tau_2$ interchanges the copies.
\item $\tau_3 = \tau_2 \circ \tau_1$, which sends $y \cup \varnothing \mapsto \varnothing \cup y^\sigma$ and $\varnothing \cup y \mapsto y \cup \varnothing$.
\end{enumerate}
The first two are involutions, while $\tau_3$ has order 4.

Then define the correspondence 
\begin{equation}\label{simple hecke}
\xymatrix{
& { H_d (W) }  \ar[dr] \ar[dl] \\
{  \Sigma_{2d}(W)  }  \ar[dr]_{ \otimes}  & &  {   \Sigma_{2d}(W)  }  \ar[dl]^{  \otimes} \\
& { \Sigma_{2d}(Y) ,} 
}
\end{equation}
where, for any $\kk$-scheme $S$,  $H_{d} (W)(S)$ is the groupoid of: 
\begin{itemize}
\item
a pair of $S$-points $(\mathcal{K}_i, \bm{a}_i) \in \Sigma_{2d}(W)$,  for $i = 0,1$,
\item
 an $S$-point $y\in W(S)$, 
 \item
  an isomorphism
\[
s : \mathcal{K}_0 ( y^{\tau_1} - y^{\tau_3} ) \iso \mathcal{K}_1 
\]
of line bundles on $W_S$ such that  $s(\bm{a}_0) = \bm{a}_1$, where 
we view $\bm{a}_0$ as a rational section of  $\mathcal{K}_0 ( y^{\tau_1} - y^{\tau_3} )$.  
\end{itemize}

This data is determined by $(\mathcal{K}_0 ,\bm{a}_0)$ and the point $y\in W(S)$, for from these we may recover the line bundle $\mathcal{K}_1 = \mathcal{K}_0 ( y^{\tau_1} - y^{\tau_3} )$ and its rational section $\bm{a}_1=\bm{a}_0$.  The condition that  $\bm{a}_1$ is a global section of $\mathcal{K}_1$, as opposed to merely a rational section, is equivalent to 
\[
\mathrm{div}(\bm{a}_0 ) + y^{\tau_1} - y^{\tau_3} \ge 0.
\]
This is in turn equivalent to the condition that the effective Cartier divisor $y^{\tau_3}$ appears in the support of $\mathrm{div}(\bm{a}_0)$.
In other words,  we may realize
\[
H_{d}(Y) \hookrightarrow \Sigma_{2d}(W) \times_k W
\]
as the closed subscheme of triples $(\mathcal{K}_0 , \bm{a}_0  , y )$ for which $y^{\tau_3}$ appears in the support of $\mathrm{div}(\bm{a}_0)$.
\begin{remark}\label{no preserve}
$H_d(W)$ does not preserve the substacks $\Sigma_{i, 2d-i}(Y) \subset \Sigma_{2d}(W)$.  In fact, it induces a correspondence from $\Sigma_{i, 2d-i}(Y)$ to 
\[\Sigma_{i + 1,2d- i - 1}(Y) \coprod \Sigma_{i-1,2d-i+1}(Y).\]
Accordingly, the correspondence $\YZ_d$ does not preserve $N_{(i,2d-i)} \subset N_d$. 
\end{remark}

\begin{proposition}\label{prop:rep one paw}
The diagram (\ref{hecke one paw}) is canonically identified with
\[
\xymatrix{
& { A_d \times_{  \Sigma_{2d}(Y) }   H_d (W) }  \ar[dr] \ar[dl] \\
{  A_d \times_{  \Sigma_{2d}(Y) }  \Sigma_{2d}(W)  }  \ar[dr]  & &  { A_d \times_{  \Sigma_{2d}(Y) }  \Sigma_{2d}(W)    }  \ar[dl]\\
& {  A_d }, 
}
\]
obtained from (\ref{simple hecke}) by base change along the map $A_d \to  \Sigma_{2d}(Y)$  in  (\ref{fundamental 2}).
\end{proposition}
\begin{proof}
It is enough to show that $\YZ_d \cong A_d \times_{  \Sigma_{2d}(Y) }   H_d (W)$.  We will use the construction in the proof of Proposition \cite[3.12]{HS}. 
Define a map 
\[
\YZ_d \to A_d \times_{  \Sigma_{2d}(Y) }   H_d (W)
\]
  as follows.  Given
\[((\mathcal{M}_{0},\mathcal{L}_0,\phi_0),( \mathcal{M}_1,\mathcal{L}_1, \phi_1), (y_0,y_1), s_0, s_1) \in \widetilde\YZ_d(S),\] we obtain from the first two pieces of data, points $(\mathcal{K}_0 , \bm{a}_0 )$ and $(\mathcal{K}_1 , \bm{a}_1 )$ of $ \Sigma_{2d}(W)(S)$.  
We have 
\[\mathcal{K}_i \cong \underline\Hom((\mathcal{L}_i, \mathcal{L}_i^\sigma), \mathcal{M}_i |_{W_S}).\]
Let $y \in W(S)$ correspond to the point $(y_0,y_1) \in (Y_0 \times_X Y)(S)$.  Then the isomorphisms $ \mathcal{L}_0(y_1) \cong \mathcal{L}_1$ and $ \mathcal{M}_0(y_0) \cong \mathcal{M}_1$ induce an isomorphism $ \mathcal{K}_0(y^{\tau_1} - y^{\tau_3})\cong \mathcal{K}_1$ sending $\bm{a}_0$ to $\bm{a}_1$.  This gives a point in $H_d(W)(S)$.  Since $(\mathcal{M}_{0},\mathcal{L}_0,\phi_0)$ and $( \mathcal{M}_1,\mathcal{L}_1, \phi_1)$ lie over the same point in $A_d$, we obtain a map 
\[
\widetilde\YZ_d \to A_d \times_{  \Sigma_{2d}(Y) }   H_d (W),
\]
 which factors through $\YZ_d$.  

Next, we construct a map in the other direction.  Suppose given an $S$-point 
\[(\Delta, \xi, \mathcal{K}_0, \bm{a}_0,\mathcal{K}_1, \bm{a}_1,y,s) \in A_d \times_{  \Sigma_{2d}(Y) } H_d (W).\]
By (\ref{fundamental 2}), we have points $(\mathcal{M}_i, \mathcal{L}_i, \phi_i) \in N_{d}(S)$ corresponding to $(\Delta, \xi, \mathcal{K}_i, \bm{a}_i)$, for $i = 0,1$.  Let's show that there are isomorphisms $\mathcal{M}_1\cong \mathcal{M}_0(y_0)$ and $\mathcal{L}_1 \cong \mathcal{L}_0(y_1)$ inducing the given isomorphism 
\[
s : \mathcal{K}_0 ( y^{\tau_1} - y^{\tau_3} ) \iso \mathcal{K}_1.
\]

Suppose first that $y \in \varnothing \coprod Y$.  Since we work modulo twisting, we may assume $\M_i = (\co_X, \M_i'')$.  If we write $\mathcal{K}_i = (\mathcal{K}_i', \mathcal{K}_i'')$, then $\LL_i = (\mathcal{K}_i')^{-1}$.  By the proof of \cite[3.12]{HS}, $\M_i''$ is a canonical descent of the line bundle $\mathcal{K}_i'' \otimes \LL_i^\sigma$ on $Y$, down to $X$.  Now, the isomorphism $\mathcal{K}_0(y^{\tau_1} - y^{\tau_3}) \simeq \mathcal{K}_1$ amounts to a pair of isomorphisms 
\begin{equation}\label{isoms}
\mathcal{K}_0'(-y_1) \simeq \mathcal{K}_1' \hspace{4mm} \mbox{and} \hspace{4mm} \mathcal{K}_0''(y_1) \simeq \mathcal{K}_1''.
\end{equation}
We deduce from this an isomorphism $\LL_0(y_1) \simeq \LL_1$.  To finish, we must produce an isomorphism $\M_0''(y_0) \simeq \M_1''$.  Now, (\ref{isoms}) gives us an isomorphism $\nu^*(\M_0''(y_0)) \simeq \nu^*\M_1''$, and the descent data defining $\M_0''$ and $\M_1''$ allows us to descend this to the desired isomorphism $\M_0''(y_0) \simeq \M_1''$.  The case $y \in Y \coprod \varnothing$ is proved similarly.   
\end{proof}

\begin{corollary}\label{prop:hk_d facts}
The stack $\YZ_d$ is a scheme, and $\gamma_0, \gamma_1 \colon \YZ_d \to N_{d}$ are finite and surjective. 
Hence, by Proposition $\ref{prop:Nsmooth}$,  if $d \geq 2g_1 - 1$ then $\dim \YZ_d = 2d - g + 1$.  
\end{corollary}

The correspondence  $\YZ_d$ induces an endomorphism 
\[
[\YZ_d] :  \beta_*\Q_\ell \to \beta_*\Q_\ell
\] 
of sheaves on $A_d$, given by the composition
\[\beta_*\Q_\ell \to \beta_*\gamma_{0*}\gamma_0^*\Q_\ell \simeq \beta_*\gamma_{0*}\Q_\ell \simeq \beta_*\gamma_{1*}\Q_\ell \to \beta_*\Q_\ell.\]
The first and last maps are induced by adjunction, using that $\gamma_0$ and $\gamma_1$ are finite.  Denote by $[\YZ_d]^r$ the $r$-fold composition of  this endomorphism with itself.  

 \begin{proposition}\label{prop:geometric trace}
 Fix an effective divisor $D\in \Div(X)$ of degree $d \geq 2g_1 - 1$, and recall the closed subscheme $A_D \subset A_d$ and the inclusion $A_D(\kk) \subset K$ of $(\ref{invariant domain})$.  The distribution $\I_r$ (\ref{I distribution def}) satisfies
 \[
 \I_r(f_D) = \sum_{ \substack{  \xi \in K  \\  \Tr_{K/F}(\xi)=1   } }\I_r(\xi , f_D),
 \]
 where 
\[
\I_r(\xi , f_D)= 
\begin{cases}
\Tr \big(      [\YZ_d]_{\bar{\xi}} ^r  \circ \mathrm{Frob}_\xi   ;\,   ( \beta_* \Q_\ell)_{\bar{\xi}}  \big)  
& \mbox{if } \xi\in A_D(\kk)\\
0 & \mbox{otherwise.}
\end{cases}
\]
Here $\bar{\xi}$ is any geometric point above $\xi:\Spec(\kk) \to A_D$.
 \end{proposition}
 
 \begin{proof}
 The proof is similar to the proof of \cite[4.7]{HS}, so we omit it.  One slight difference is the analogue of \cite[Lemma 6.11(1)]{YZ}.  Specifically, we must show that the map $ \pi \colon \mathrm{Hk}_T^\mu \to \mathrm{Hk}_{G}^\mu(y)$ discussed there remains a regular local immersion if we replace $\mathrm{Hk}_T^\mu$ with $\mathrm{Hk}_A^\mu$.  The tangent complex computations are similar except now both tangent complexes have 1-dimensional $H^0$.  In particular, it is no longer true that $\mathrm{Hk}_{G}^\mu(y)$ is Deligne-Mumford in the neighborhood of a point in the image of $\pi$.  Nevertheless, the induced map on $H^0$ (resp.\ $H^1$) is an isomorphism (resp.\ injection), which is enough to conclude that the map is indeed a regular local immersion of algebraic stacks.  
 \end{proof}

\section{Comparison of traces}

\subsection{Representations}
For now, $\ell$ is any prime and all representations are over $\Q_\ell$.  

Let $n \geq 1$ be an even integer.  For each $0 \leq i \leq n$, view $S_i \times S_{n-i}$ as the subgroup of $S_n$ preserving the subset $\{1,\ldots, i\}$.  Let $\bm{1}$ denote the trivial representation $\Q_\ell$.  Then the induced representation 
\[V_i = \mathrm{Ind}_{S_i \times S_{n-i}}^{S_{n}} \bm{1}\] is the permutation representation for the $S_{n}$-action on the ways of dividing the integers $\{1,\ldots, n\}$ into two bins of size $i$ and $n-i$.  

Let $V = \bigoplus_{i = 0}^{n} V_i$.  By the combinatorial description of $V_i$, we see that 
\[V \simeq (\Q_\ell^2)^{\otimes n},\]  where the action of $S_{n}$ on $(\Q_\ell^2)^{\otimes n}$ is by permuting coordinates.

Recall that the irreducible representations $\rho_\lambda$ of $S_n$ are indexed by partitions $\lambda \vdash n$. Write $\lambda  = [\lambda_1,\ldots,\lambda_k]$ with  $\lambda_1 \geq \lambda_2 \geq \cdots \geq \lambda_k \geq 1$ and $\sum_i \lambda_i = n$.  We will also interpret the symbol $[n,0]$ to mean $[n]$.  Write 
\[V = \bigoplus_{\lambda \vdash n} V_\lambda,\]
\[V_i = \bigoplus_{\lambda \vdash n} V_{i,\lambda}\] 
where $V_\lambda$ and $V_{i,\lambda}$ are the $\rho_\lambda$-isotypic components of $V$ and $V_i$, respectively. 

\begin{proposition}\label{LRrule}
If $\lambda \vdash n$,  then  
\begin{enumerate}[$(a)$]
\item $V_{i,\lambda} \neq 0$ if and only if $\lambda = [k,n-k]$ with $n-k\leq i \leq k$;
\item if $V_{i,\lambda} \neq 0$, then it has multiplicity $1$, i.e.\ $V_{i,\lambda} \simeq \rho_\lambda$.
\end{enumerate}   
\end{proposition}

\begin{proof}
This is a simple application of the Littlewood-Richardson rule.
\end{proof}

Note that the diagonal action of $\GL_2(\Q_\ell)$ on $V$ commutes with the $S_n$ action.  There is also a lie algebra action of $\mathfrak{sl}_2$ on $V \simeq (\Q_\ell^2)^{\otimes 2n}$, given by  
\begin{align*}
X \cdot v_1 \otimes &v_2 \otimes \cdots \otimes v_n  \\
&= Xv_1 \otimes \cdots \otimes v_n + v_1 \otimes Xv_2 \otimes \cdots \otimes v_n +  \ldots + v_1 \otimes \cdots \otimes Xv_n.
\end{align*}
This too commutes with the $S_n$ action.  By Schur-Weyl duality, each $V_\lambda$ is isomorphic to $M_\lambda \otimes \rho_\lambda$, for some irreducible $\mathfrak{sl}_2$-representation $M_\lambda$ (on which $S_n$ acts trivially).  

\begin{corollary}\label{isotypic component}
If $\lambda = [k,n-k]$, then $V_\lambda \simeq \mathrm{Sym}^{2k - n}\Q_\ell^2 \otimes \rho_\lambda$.
\end{corollary}
\begin{proof}
By Proposition \ref{LRrule}, we have $\dim V_\lambda = (2k-n+1)\dim\rho_\lambda.$   
Since $\mathrm{Sym}^{2k-n}\Q_\ell^2$ is the unique irreducible representation of $\mathfrak{sl}_2$ of dimension $2k-n + 1$, we must have $M_\lambda \simeq \mathrm{Sym}^{2k - n}\Q_\ell^2$.
\end{proof}

Let $\{e_+,e_-\}$ be a basis for $\Q_\ell^2$. For $\epsilon \in \{\pm\}^n$, let $e_{\epsilon}$ be the corresponding basis element of $V$.  Also let $\epsilon_i$ be the element 
\[(+,+,\ldots,+,-,+,\ldots,+,+) \in \{\pm\}^n,\]
with a minus sign in the $i$-th coordinate. 

Let $H \colon V \to V$ be the $\Q_\ell$-linear map determined by  
\[H(e_\epsilon)  =  \sum_{i = 1}^n e_{\epsilon\epsilon_i}.\]
\begin{lemma}
$H$ commutes with the $S_n$ action on $V$.
\end{lemma}

In fact, the map $H$ is simply the action of $(\begin{smallmatrix} 0 & 1 \\ 1 & 0\end{smallmatrix}) \in \mathfrak{sl}_2$ on $V$.  
Now suppose $\lambda = [k,n-k]$.  Since $H$ commutes with the $S_n$-action, it acts on $V_\lambda$ as $H_\lambda \otimes 1$, for some $H_\lambda \in \End(\mathrm{Sym}^{2k}\Q_\ell^2)$.    
\begin{corollary}\label{eigs}
Let $\lambda = [k,n-k]$.  Then the eigenvalues of $H_\lambda$ are 
\[\{-2k,-2k+2,\ldots,-2,0,2,\ldots, 2k-2,2k\}\] each appearing with multiplicity one.
\end{corollary}
\begin{proof}
The matrix $(\begin{smallmatrix} 0 & 1 \\ 1 & 0\end{smallmatrix})$ is usually denoted $e + f$ in the representation theory of $\mathfrak{sl}_2$.  It is conjugate to the matrix $(\begin{smallmatrix} 1 & 0 \\ 0 & -1\end{smallmatrix})$, often denoted $h$.  Thus, the characteristic polynomial of $H_\lambda$ is the same as that of $h$ acting on the space $\mathrm{Sym}^{2k}\Q_\ell^2$.  The natural basis for  $\mathrm{Sym}^{2k}\Q_\ell^2$ are eigenvectors for $h$ with eigenvalues precisely the ones stated.  
\end{proof}
\subsection{Local systems}
Now let $\ell$ be a prime different from char $\kk$.  Recall the split double cover $W = Y \coprod Y$ of $Y$. Let  $U_{2d}(Y)  \subset Y^{2d}$  be the open subscheme parametrizing  $2d$-tuples of \emph{distinct} points on $Y$,
and let $U_{2d}(W) \subset W^{2d}$ be its preimage under the morphism $W^{2d} \to Y^{2d}$. 
Thus we have a cartesian diagram
\[
\xymatrix{
{  U_{2d}(W) } \ar[r] \ar[d] &  { W^{2d}  }  \ar[d] \\
{  U_{2d}(Y)  }  \ar[r]  & {  Y^{2d} },
}
\]
in which the horizontal arrows are open immersions with dense image, and the vertical arrows are finite \'etale.  
Both $W^n$ and $U_{2d}(W)$ are disconnected.  The connected components are 
\[W^{2d} = \coprod_{i = 0}^{2d} (Y^i \times Y^{2d-i})\]
\[U_{2d}(W) = \coprod_{i = 0}^{2d} U_{i,2d-i}(W).\]

Taking the appropriate quotients, and using the isomorphisms of (\ref{symmetric id}), we obtain a cartesian diagram
\begin{equation}\label{unramified stratum}
\xymatrix{
{ S_{i} \times S_{2d-i} \backslash U_{i,2d-i}(W) } \ar[r] \ar[d]_{b_i} &  {   \Sigma_{i,2d-i}(W) }  \ar[d]^{\Nm} \\
{   S_{2d} \backslash U_{2d}(Y)  }  \ar[r]_u  & {   \Sigma_{2d}(Y) ,}
}
\end{equation}
in which the horizontal arrows are open immersions, the vertical arrows are finite, and $b_i$ is \'etale.

The Galois cover  
\[
U_{i,2d-i}(W) \to S_{2d} \backslash U_{2d}(Y),
\]
has group $S_{2d}$, and the local system $b_{i*}\Q_\ell$ on $S_{2d}\backslash U_{2d}(Y)$  corresponds to the induced representation $V_i$ from the previous section. 
We define 
\[b_*\Q_\ell:= \bigoplus_{i = 0}^{2d} b_{i*}\Q_\ell.\] 
This is a local system  on $S_{2d}\backslash U_{2d}(Y)$ corresponding to the $S_{2d}$-representation $V \simeq (\Q_\ell^2)^{\otimes 2d}$.
If $\lambda \vdash n$, we let $\widetilde L_\lambda$ be the local system on $S_{2d}\backslash U_{2d}(Y)$ corresponding to $\rho_\lambda$.  

Define  $A_d^\circ$ as the cartesian product
\[
\xymatrix{
{  A_d^\circ } \ar[r]^v \ar[d]_\pi &   {  A_d }   \ar[d]^{\nu^\sharp}  \\
{      S_{2d} \backslash U_{2d}(Y)    }  \ar[r]_u&  {  \Sigma_{2d}(Y), } 
}
\]
and set $L_\lambda = v_*\pi^*\widetilde L_\lambda$.  Recall from \S2 the maps $\beta_i \colon N_{(i,2d-i)} \to A_d$, for each $i = 0,\ldots, 2d$.   

\begin{proposition}\label{beta i decomp}
For each $i \in \{0,1,\ldots, d\}$, there is an isomorphism 
\[\beta_{i*}\Q_\ell \simeq (\beta_{2d-i})_*\Q_\ell\] 
and a decomposition
\[\beta_{i*}\Q_\ell = \bigoplus_{k = 0}^i L_{[2d-k,k]}.\]
\end{proposition}
\begin{proof}
By Proposition \ref{LRrule}, we have 
\[b_{i*}\Q_\ell = \bigoplus_{k=0}^i \widetilde L_{[2d-k,k]}.\]
On the other hand, it follows from proper base change and the smoothness of $N_{(i,2d-i)}$ that
\[\beta_{i*}\Q_\ell \simeq v_*\pi^*b_{i*}\Q_\ell;\]
c.f.\ \cite[Prop.\ 5.3]{HS}. The proposition now follows.
\end{proof}

\begin{proposition}\label{beta decomp}
There is a decomposition 
\[\beta_*\Q_\ell = \bigoplus_{k = 0}^d \left(\mathrm{Sym}^{2d-2k}\Q_\ell^2\otimes L_{[2d-k,k]}\right),\]
and the endomorphism $[\YZ_d]$ stabilizes each summand.  The action of $[\YZ_d]$ on $\mathrm{Sym}^{2d-2k}\Q_\ell^2\otimes L_{[2d-k,k]}$ is of the form $H_k \otimes 1$, for some $H_k$ having characteristic polynomial 
\[\det(t - H_k) = \prod_{j = k-d}^{d-k}(t - 2j).\]
\end{proposition}
\begin{proof}
The proof of the first part is as in Proposition \ref{beta i decomp}, this time using the decomposition 
\[V= \bigoplus_{k=0}^d V_{[2d-k,k]}\]
and Corollary \ref{isotypic component}.
The $S_{2d}$-map $H \colon V \to V$ corresponds to a map $H \colon b_*\Q_\ell \to b_*\Q_\ell$ of local systems. By Proposition \ref{prop:rep one paw}, the map $[\YZ_d]|_{A_d^\circ}$ is identified with $\pi^*H$.  Thus, the second part of the proposition follows from Corollary \ref{eigs}.    
\end{proof}

\subsection{The key identity}
Fix an auxiliary prime $\ell \neq \mathrm{char} (\kk)$, and define $\mathscr{H}_\ell=\mathscr{H}\otimes\Q_\ell$, the $\ell$-adic analogue of the $\Q$-algebra $\mathscr{H}$ of  \S \ref{ss:basic automorphic}.

The Hecke algebra $\mathscr{H}_\ell$ acts on the $\ell$-adic cohomology group
\[
V = H_c^{2r}(\Sht^\mu_{G} \otimes_\kk \bar \kk, \Q_\ell)(r),
\]
as in \cite[\S 7.1]{YZ}.   The cycle class map $\mathrm{cl} \colon \Ch_{c,r}(\Sht_{G}^r) \to V$
 is $\mathscr{H}$-equivariant, and  the cup product   
 \begin{equation}\label{cup}
 \langle\cdot,\cdot\rangle:  V\times V \to \Q_\ell
 \end{equation}
 pulls back to  the intersection pairing on the Chow group.

Recalling the map $\mathscr{H} \to \Q[\Pic_X(\kk)]^{\iota_\Pic}$ appearing in (\ref{eisenstein def}), define
\begin{align*}
\widetilde{\mathscr{H}}_\ell 
& = \mathrm{Image}  \big(  \mathscr{H}_\ell \to \End_{\Q_\ell} (V) \times \End_{\Q_\ell}(\mathscr{A}_\ell) \times \Q_\ell[\Pic_X(\kk)]^{\iota_\Pic}  \big)   \\
\overline{\mathscr{H}_\ell} 
& = \mathrm{Image}  \big(  \mathscr{H}_\ell \to \End_{\Q_\ell} (V) \times  \Q_\ell[\Pic_X(\kk)]^{\iota_\Pic}  \big) \\
\mathscr{H}_{\mathrm{aut},\ell} 
& = \mathrm{Image}  \big(  \mathscr{H}_\ell \to \End_{\Q_\ell}(\mathscr{A}_\ell) \times \Q_\ell[\Pic_X(\kk)]^{\iota_\Pic}  \big).
\end{align*}
These are  finite type $\Q_\ell$-algebras, related by surjections
\[
\xymatrix{
& { \widetilde{\mathscr{H}}_\ell } \ar[dl] \ar[dr] \\
{  \overline{\mathscr{H}_\ell}  }  \ar[dr] &  &  {  \mathscr{H}_{\mathrm{aut},\ell}  }  \ar[dl] \\
& {   \Q_\ell[\Pic_X(\kk)]^{\iota_\Pic}  .  }   
}
\]
Recalling the $\Q$-algebra $\mathscr{H}_\mathrm{aut}$ of \S \ref{ss:basic automorphic}, there is a canonical isomorphism 
\[
 \mathscr{H}_\mathrm{aut} \otimes\Q_\ell \iso \mathscr{H}_{\mathrm{aut},\ell} .
\]
For any $f\in \mathscr{H}$, the function $\J(f,s)$ of (\ref{J distribution def}) is a Laurent polynomial in $q^s$ with rational coefficients.
Setting 
\[
\J_r(f) = ( \log q)^{-r}  \frac{d^r}{ds^r}  \J(f ,s)  \big|_{s=0},
\]
we obtain a linear functional $\J_r : \mathscr{H} \to \Q$.  The following result shows that this agrees with the linear functional 
$\I_r : \mathscr{H} \to \Q$ defined by (\ref{I distribution def}).

\begin{theorem}\label{prop:key identity}
If $f \in \mathscr{H}$, then $\I_r(f) = \J_r(f)$.  
Moreover, the  $\Q_\ell$-linear extensions  of $\I_r$ and $\J_r$ to $\mathscr{H}_\ell \to \Q_\ell$   factor through  $\widetilde{\mathscr{H}}_\ell$.
\end{theorem}

\begin{proof}
The compatibility of the cup product pairing (\ref{cup}) with the intersection pairing on the Chow group implies that 
the $\Q_\ell$-linear extension $\I_r : \mathscr{H}_\ell \to \Q_\ell$ factors through  $\overline{\mathscr{H}_\ell}$. 
The final claim of Proposition \ref{J pi decomp} implies that the $\Q_\ell$-linear extension $\J_r : \mathscr{H}_\ell \to \Q_\ell$ factors through  $\mathscr{H}_{\mathrm{aut}, \ell}$.  It follows that both   $\I_r$ and $\J_r$  factor through the quotient 
$ \widetilde{\mathscr{H}}_\ell$.

It  remains to prove that $\I_r(f) = \J_r(f)$ for all $f\in \mathscr{H}$.
Assume first that $f=f_D$ for some effective divisor $D \in \mathrm{Div}(X)$ of degree $d\ge 2 g_1-1$.  
For each $\xi \in A_D(\kk)$, define
\[\I_r(\xi, f_D) = \Tr\big(  [\YZ_d]^r_{\bar{\xi}} \circ \mathrm{Frob}_\xi  ;  ( \beta_*\Q_\ell)_{\bar\xi}  \big),\]
and
\[\J_r(\xi, f_D) = 2^r\sum_{i = 0}^{2d}(i-d)^r\Tr\big(\mathrm{Frob}_\xi  ;  ( \beta_{(i,2d-i)*}\Q_\ell)_{\bar\xi}  \big) \]
By Proposition \ref{prop:geometric trace} and Theorem \ref{prop:geometric orbital}, 
\[\I_r(f_D) = \sum_{\xi \in A_D(\kk)} \I_r(\xi,f_D).\]
\[\J_r(f_D) = \sum_{\xi \in A_D(\kk)} \J_r(\xi,f_D),\]
so it suffices to show that $\I_r(\xi,f_D) = \J_r(\xi,f_D)$ for all $\xi$. 

There is a decomposition
\[\beta_*\Q_\ell = \bigoplus_{i = 0}^{2d} \beta_{(i,2d-i)*}\Q_\ell,\] but the endomorphism $[\YZ_d]$ of $\beta_*\Q_\ell$ does not preserve this decomposition; see Remark \ref{no preserve}.  We instead consider the decomposition 
\[\beta_*\Q_\ell = \bigoplus_{k = 0}^d \left(\mathrm{Sym}^{2d-2k}(\Q_\ell^2)\otimes L_{[2d-k,k]}\right),\]
of Proposition \ref{beta decomp}, which has the property that $[\YZ_d]$ takes the form $H_k \otimes 1$ on each summand.  We compute
\begin{align*}
\I_r(\xi, f_D) &= \Tr\big(  [\YZ_d]^r_{\bar{\xi}} \circ \mathrm{Frob}_\xi  ;  ( \beta_*\Q_\ell)_{\bar\xi}  \big)\\
&=\sum_{k = 0}^d \Tr\left(  H_k^r \otimes \mathrm{Frob}_\xi  ;  \mathrm{Sym}^{2d-2k}(\Q_\ell^2)\otimes L_{[2d-k,k]_{\bar\xi}}  \right)\\
&=\sum_{k = 0}^d \Tr\left(H_k^r\right) \Tr\left(\mathrm{Frob}_\xi  ;  L_{[2d-k,k]_{\bar\xi}}  \right) \\
&=\sum_{k = 0}^d \sum_{j = k-d}^{d-k} (2j)^r  \Tr\left(\mathrm{Frob}_\xi  ;  L_{[2d-k,k]_{\bar\xi}}  \right)\\
&=2^{r+1}\sum_{k = 0}^d \sum_{j = 0}^{d-k} j^r  \Tr\left(\mathrm{Frob}_\xi  ;  L_{[2d-k,k]_{\bar\xi}}  \right)
\end{align*} 
where the second to last equality follows from the last statement in Proposition \ref{beta decomp}.

On the other hand, using Proposition \ref{beta i decomp} we compute 
\begin{align*}
\J_r(\xi, f_D) &= 2^r\sum_{i = 0}^{2d}(d-i)^r\Tr\big(\mathrm{Frob}_\xi  ;  ( \beta_{(i,2d-i)*}\Q_\ell)_{\bar\xi}  \big) \\
&= 2^{r+1}\sum_{i = 0}^{d}(d-i)^r\Tr\big(\mathrm{Frob}_\xi  ;  ( \beta_{(i,2d-i)*}\Q_\ell)_{\bar\xi}  \big) \\
&=2^{r+1}\sum_{i = 0}^{d}(d-i)^r\sum_{k = 0}^i \Tr\big(\mathrm{Frob}_\xi  ;  \left(L_{[2d-k,k]}\right)_{\bar\xi}  \big)\\
&=2^{r+1}\sum_{k = 0}^{d}\sum_{j = 0}^{d-k}j^r \Tr\big(\mathrm{Frob}_\xi  ;  \left(L_{[2d-k,k]}\right)_{\bar\xi}  \big).
\end{align*}
We conclude that $\I_r(\xi,f_D) = \J_r(\xi,f_D)$, and hence $\I_r(f_D) = \J_r(f_D)$ for all effective $D$ of degree $d \geq 2g_1 - 1$ . 
 
 The proof of \cite[Theorem 9.2]{YZ} shows that the image of $\mathscr{H}_\ell \to \widetilde{\mathscr{H}}_\ell$  is generated as $\Q_\ell$-vector space by the images of $f_D\in \mathscr{H}$ as $D$ ranges over all effective divisors on $X$ of degree $d \ge 2 g_1-1$.  Therefore $\I_r=\J_r$.
\end{proof}

\subsection{Proof of main theorems}

The main theorems of the introduction follow from Theorem \ref{prop:key identity}, by modifying the formal arguments of \cite[\S 5]{HS}.  The additional subtlety in our context is that the intersection pairing appearing in the definition of $\I_r(f)$ depends on the auxiliary integer $d$, which is itself a function of $f$.  For the convenience of the reader, we spell out the argument below.   

According to  \cite[(9.5)]{YZ}, there is a canonical $\Q_\ell$-algebra decomposition
\begin{equation}\label{ell hecke decomp}
\widetilde{\mathscr{H}}_\ell = 
\widetilde{\mathscr{H}}_{ \ell, \mathrm{Eis}}  \oplus
 \big( \bigoplus_\mathfrak{m}   \widetilde{\mathscr{H}}_{ \ell,  \mathfrak{m} }   \big),
\end{equation}
where $\mathfrak{m}$ runs over the finitely many maximal ideals $\mathfrak{m} \subset \widetilde{\mathscr{H}}_\ell$ that do not contain the 
kernel of the projection  
\begin{equation}\label{ell satake}
\widetilde{\mathscr{H}}_\ell  \to \Q_\ell[\Pic_X(\kk)]^{\iota_{\Pic}} .
\end{equation}
For each such $\m$ the localization  $ \widetilde{\mathscr{H}}_{\ell,\mathfrak{m}}$ is a finite (hence Artinian) $\Q_\ell$-algebra.
If we denote by $E_\m$ its residue field, then Hensel's lemma implies that the quotient map $\widetilde{\mathscr{H}}_{\ell,\mathfrak{m}} \to E_\mathfrak{m}$ admits a unique section, which makes $\widetilde{\mathscr{H}}_{\ell, \mathfrak{m}}$ into an Artinian local $E_\mathfrak{m}$-algebra.

The decomposition (\ref{ell hecke decomp}) induces a decomposition  of $\widetilde{\mathscr{H}}_\ell$-modules
\begin{equation}\label{cohomological decomposition}
V = V_\Eis \oplus \big(\bigoplus_\m V_\m\big),
\end{equation}
in which each localization $V_\m$ is a finite-dimensional $E_\m$-vector space.
It follows from \cite[Corollary 7.15]{YZ} that this decomposition is orthogonal with respect to the cup product pairing.   
Moreover, the self adjointness of the action of $\mathscr{H}_\ell$ with respect to the cup product pairing (\ref{cup})
implies that there is a unique symmetric $E_\m$-bilinear pairing 
\[
\langle \cdot , \cdot \rangle_{E_\m} : V_\m \times V_\m \to E_\m
\]
such that $\Tr_{E_\m/\Q_\ell} \langle \cdot , \cdot \rangle_{E_\m}= \langle\cdot,\cdot\rangle$.

Define
$
[\Sht_{T}^\mu]_\m \in V_\m
$ 
to be the projection  of the cycle class $\mathrm{cl}( [\Sht_{T}^\mu]) \in V$. Next, define 
$
[\Sht_{A}^\mu]_\m \in V_\m
$ 
to be the projection of $\mathrm{cl}( [\Sht_{A}^{\mu, \leq d+r}]) \in V$, where $d \in \Z \subset \mathcal{D}$ is any integer such that  
\[[\Sht_{T}^\mu]_\m \in H^{2r}_c(\Sht_G^{\mu,\leq d} \otimes_\kk \bar\kk, \Q_\ell)(r) \subset V.\]  
We may form the intersection pairing
\[
\langle [\Sht_{A}^\mu]_\m, [\Sht_{T}^\mu]_\m \rangle_{E_\m} \in E_\m.
\]

Some maximal ideals $\mathfrak{m} \subset \widetilde{\mathscr{H}}_\ell$ appearing in (\ref{ell hecke decomp}) are attached to cuspidal automorphic forms.  Fix an  unramified cuspidal automorphic representation $\pi \subset \mathcal{A}_\cusp(G)$. As 
in \S \ref{ss:basic automorphic}, such a representation determines a homomorphism 
\[
\mathscr{H}_\mathrm{aut} \to \mathscr{H}_\cusp \map{\lambda_\pi} \C
\] 
whose image is a number field $ E_\pi$.  The induced map
\[
\widetilde{ \mathscr{H} }_\ell \to  \mathscr{H}_{\mathrm{aut},\ell} \map{\lambda_\pi} E_\pi \otimes \Q_\ell \iso \prod_{ \mathfrak{l} \mid \ell} E_{\pi ,\mathfrak{l}},
\]
determines,  for every prime  $\mathfrak{l}\mid \ell$ of $E_\pi$,  a surjection
$\lambda_{\pi,\mathfrak{l}}   :    \widetilde{\mathscr{H}}_\ell \to  E_{\pi,\mathfrak{l}}$ whose kernel is one of those maximal ideals 
\begin{equation}\label{cuspidal points}
\mathfrak{m}=\ker( \lambda_{\pi,\mathfrak{l}} )
\end{equation}
appearing in the decomposition (\ref{cohomological decomposition}).  This is a consequence of the isomorphism (\ref{automorphic decomp}).

Recalling the period integrals $\mathscr{P}_A$ and $\mathscr{P}_T$ of \S \ref{ss:spectral}, for every unramified cuspidal automorphic representation $\pi \subset\mathcal{A}_\cusp(G)$ define 
\[
C(\pi,s) =   \frac{   \mathscr{P}_A (  \phi ,s) \mathscr{P}_T(\overline{\phi} ,\eta)}{   \langle \phi, \phi \rangle_\mathrm{Pet} } .
\]
Here, $\phi$ is any nonzero vector in $\pi^U$.

\begin{proposition}\label{prop:period reciprocity}
The complex number 
\[
C_r(\pi)  =  ( \log q)^{-r} \cdot  \frac{d^r}{ds^r}  C (  \pi ,s) \big|_{s=0}
\]  
satisfies $C_r(\pi)^\sigma= C_r(\pi^\sigma)$ for all $\sigma \in \Aut(\C/\Q)$.  In particular, it lies in  $E_\pi$.
\end{proposition}

\begin{proof}
As in \cite[Prop.\ 5.7]{HS}.
\end{proof}

\begin{theorem}\label{main}
Let $\mathfrak{m} \subset \widetilde{\mathscr{H}}_\ell$ be a maximal ideal that does not contain the kernel of (\ref{ell satake}).
\begin{enumerate}
\item
If $\mathfrak{m}$ is of the form (\ref{cuspidal points}) for an unramified cuspidal automorphic representation $\pi$ and a place $\mathfrak{l}\mid \ell$ of $E_\pi$, the equality
\[
\langle [\Sht_{A}^\mu]_\m , \,   [\Sht_{T}^\mu]_\m \rangle_{E_\m} = C_r(\pi)
\]
holds  in $E_\m=E_{\pi ,\mathfrak{l}}$.
\item
If $\mathfrak{m}$ is not of the form (\ref{cuspidal points}) then 
\[
\langle [\Sht_{A}^\mu]_\m , \,   [\Sht_{T}^\mu]_\m \rangle_{E_\m} =0.
\]
\end{enumerate}
\end{theorem}

\begin{proof}
Given Proposition \ref{prop:key identity}, the proof follows that of  \cite[Theorem 1.6]{YZ}.
\end{proof}


\subsection{The proof of Theorem \ref{thm:main intro}}
\label{ss:main proofs}


As in the introduction, let $[\Sht_{T}^\mu]$ be the pushforward of the fundamental class under    
\[
\theta_T^\mu \colon \Sht_{T}^\mu \to \Sht_G^\mu,
\] 
and let 
$
\widetilde W_T \subset \Ch_{c,r}(\Sht_G^\mu)
$ 
be the  $\mathscr{H}$-submodule it generates.   

Let $d \in \Z \subset \mathcal{D}$ be any large enough integer so that the finite dimensional subspace 
\[\big( \bigoplus_\mathfrak{m}   \widetilde{\mathscr{H}}_{ \ell,  \mathfrak{m} }   \big) \cdot \mathrm{cl}([\Sht_T^\mu]) \subset H^{2r}(\Sht_G^\mu \otimes_\kk \bar\kk, \Q_\ell(r))\]
is supported on $\Sht_G^{\mu, \leq d-r}$; c.f.\  (\ref{ell hecke decomp}). 
Then define $[\Sht_{A}^\mu]$ to be the pushforward of the fundamental class under    
\[
\theta_A^{\mu} \colon \Sht_{A}^{\mu, \leq d} \to \Sht_G^\mu.
\] 
Let 
$
\widetilde W_A \subset \Ch_{c,r}(\Sht_G^\mu)
$ 
be the  $\mathscr{H}$-submodule generated by $[\Sht_{A}^\mu]$.   

Define quotients  
\begin{align*}
W_T & = \widetilde{W}_T/  \{ c \in \widetilde{W}_T : \langle c, \widetilde{W}_A \rangle =0 \}  \\
W_A & = \widetilde{W}_A/  \{ c \in \widetilde{W}_A : \langle c, \widetilde{W}_T\rangle =0 \},
\end{align*}
so that the intersection pairing  descends to  $\langle \cdot\, , \cdot\rangle :  W_A \times W_T \to \Q$.

\begin{proposition}\label{prop:automorphic chow}
The actions of $\mathscr{H}$ on $W_T$ and $W_A$ factor through the quotient 
\[
\mathscr{H} \to \mathscr{H}_\mathrm{aut} \iso \mathscr{H}_\mathrm{cusp} \times \Q[\Pic_X(\kk)]^{\iota_\Pic}
\]
defined in \S \ref{ss:basic automorphic}.
\end{proposition}

\begin{proof}
By Proposition \ref{J pi decomp} the distribution $\J_r(f)$ only depends on the image of $f$ under $\mathscr{H} \to \mathscr{H}_\mathrm{aut}$.
By Proposition \ref{prop:key identity} the same is true of the distribution $\I_r(f)$ defined by (\ref{I distribution def}), and the claim follows as in \cite[Cor.\ 9.4]{YZ}. 
\end{proof}

It follows from the discussion of \S \ref{ss:basic automorphic} that $\mathscr{H}_{\cusp,\R}= \mathscr{H}_\cusp\otimes_\Q\R$ is isomorphic to a product of copies of $\R$,  indexed by the unramified cuspidal automorphic representations $\pi$.  
For each such $\pi$, let $e_\pi \in \mathscr{H}_{\cusp,\R}$ be the corresponding idempotent.  
Using Proposition \ref{prop:automorphic chow}, these idempotents induce a decomposition, for $* \in \{A,T\}$,
\[
W_* (\R) =   W_{*,\mathrm{cusp}} \oplus W_{*,\mathrm{Eis}}  = \left(\bigoplus _\pi W_{*, \pi}\right) \oplus W_{*,\mathrm{Eis}} 
\] 
with sum over unramified cuspidal  $\pi$, and where $W_{*, \pi} \subset W_*(\R)$ is the $\lambda_\pi$-eigenspace of $\mathscr{H}$.

The following is Theorem \ref{thm:main intro} of the introduction.

\begin{theorem}\label{thm:main chow}
For $* \in \{A,T\}$, let $[\Sht_{*}^\mu]_\pi$ denote the projections of the images of $[\Sht_{*}^\mu]$ to the summand $W_{*,\pi}$. Then 
\[
\langle   [\Sht_{A}^\mu]_\pi, [\Sht_{T}^\mu]_\pi \rangle  = C_r(\pi). 
\]  
\end{theorem}

\begin{proof}
This follows from Theorem \ref{main}, as in \cite[Thm.\ 5.10]{HS}.
\end{proof}

Theorems \ref{ratio} and \ref{consequence} follow from Theorem \ref{thm:main intro}, as explained in the Introduction.



\bibliographystyle{amsalpha}

\providecommand{\bysame}{\leavevmode\hbox to3em{\hrulefill}\thinspace}
\providecommand{\MR}{\relax\ifhmode\unskip\space\fi MR }
\providecommand{\MRhref}[2]{%
  \href{http://www.ams.org/mathscinet-getitem?mr=#1}{#2}}

\end{document}